\documentclass{article}

 \usepackage{times,relsize}
\usepackage{amsmath,amsthm, amssymb,bbm}

\newcommand{\smldots}{%
\mathinner{{\ldotp}{\ldotp}{\ldotp}}%
}
\newcommand{\lp}{L^p(\mathbb{R}^n)}
\newcommand{\R}{\mathbb{R}}
\newcommand{\C}{\mathbb{C}}

\newcommand{\Cv}{\mathcal{C}}

\newcommand{\Sp}{\mathbb{S}}
\newcommand{\M}{\mathcal{M}_{k,n}}
\newcommand{\lqm}{L^q(\mathcal{M}_{k,n})}
\newcommand{\lqnm}{L^{n+1}(\mathcal{M}_{k,n})}
\newcommand{\Tx}{\mathcal{T}_{x'_0\!,\ldots,\!x'_n}\!}
\newcommand{\voln}{\text{det}}
\newcommand{\volk}{\Delta}
\newcommand{\one}{\mathbbm{1}} 
\newcommand{\tolrho}{\epsilon_\rho} 
\newcommand{\tolb}{\epsilon_b}
\newcommand{\delrho}{\delta_\rho}  
\newcommand{\delb}{\delta_b} 
\newcommand{\sgn}{\text{sgn}} 

\newtheorem{thm}{Theorem}
\newtheorem{defn}{Definition}
\newtheorem{prop}{Proposition}
\newtheorem{lem}{Lemma}

\newtheorem{cor}{Corollary}
\newtheorem*{nota}{Notation}

\begin{document}

\title {Uniqueness of extremizers for an endpoint inequality of the $k$-plane transform}
\author{Taryn C. Flock\thanks{The author was supported in part by NSF grant DMS-0901569}  }
\date{}

\maketitle

\begin{abstract}
    The $k$-plane transform is a bounded operator from $\lp$ to $L^q$ of the Grassmann manifold of all affine $k$-planes  in 
    $\R^n$ for certain exponents depending on $k$ and $n$. In the endpoint case $q=n+1$, we identify all extremizers of the 
	associated inequality for the general $k$-plane transform.
\end{abstract}

\begin{section}{Introduction}
	Let $\mathcal{G}_{k,n}$ be the Grassmann manifold of all $k$-planes in $\R^n$ passing through the origin and let 
	$\mathcal{M}_{k,n}$ be the Grassmann manifold of all affine $k$-planes in $\R^n$.  Parameterize $\M$ by $(\theta,y)$ where 
	$\theta\in\mathcal{G}_{k,n}$ and $y$ is in the $(n-k)$-dimensional subspace orthogonal to $\theta$, so that $(\theta,y)$ represents the affine $k$-plane, $\theta$ translated by $y$. Equip $\M$ with the product measure formed by pairing the unique Haar probability measure on $\mathcal{G}_{k,n}$, denoted $d\gamma(\theta)$, and Lebesgue measure on the $(n-k)$-dimensional subspace orthogonal to $\theta$, denoted $d\lambda_{\theta^\perp}(y)$. Let $d\lambda_\theta$ is Lebesgue measure on the $k$-plane $\theta$. \\ 
	\indent The $k$-plane transform in $\R^n$  is given by 
		$$T_{k,n}f(\theta,y)=\int_{x\in\theta}f(x+y)\;d\lambda_\theta(x). $$
	 When $k=n-1$ this is the Radon transform and when $k=1$ it is the X-ray transform. This operator is also called the $k$-plane transform in Euclidean space.  \\
	\indent The $k$-plane transform satisfies several inequalities (see \cite{C84} and \cite{BL97}). We are concerned with the $L^p(\R^n)$-$L^q(\M)$ inequality in the case that $q=n+1$, $p=\frac{n+1}{k+1}$: 
	\begin{equation}\label{main} 
		 \left( \int_{\mathcal{G}_{k,n}}\!\!\int_{\theta^\perp}|T_{k,n}f(\theta,y)|^{n+1}\;d\lambda_{\theta^\perp}(y)d\gamma(\theta) \right)^{1/(n+1)}
		 \leq A \|f\|_{L^{\frac{n+1}{k+1}}(\R^n)}.
	\end{equation} 
	This is an endpoint inequality in the sense that the $L^p(\R^n)$-$L^q(\M)$ inequalities satisfied by $T_{k,n}$ are, up to constant factors, precisely those that follow from interpolating \eqref{main} and the trivial $L^1(\R^n)$-$L^1(\M)$ inequality. 
	\begin{defn}  A function $f\in L^p(\R^n)$ is an extremizer of \eqref{main} if it has nonzero norm and satisfies 
		\begin{equation*} 
			\frac{\|T_{k,n}f\|_{\lqm} }{\|f\|_{L^p(\R^n)}}= \sup_{\{g:\|g\|_{L^p(\R^n)}\neq0\}} \frac{ \|T_{k,n}g\|_{\lqm}}{\|g\|_{L^p(\R^n)}}
		\end{equation*}
for $p=\frac{n+1}{k+1}$ and $q=n+1$. 
	\end{defn}
Extremizers and optimal constants have been determined for some of the most fundamental $L^p$ inequalities of
Fourier and real analysis. Among such achievements is the celebrated work of Lieb  \cite{HLS} on the Hardy-Littlewood-Sobolev inequality. In \cite{BL97}, Baernstein and Loss conjectured that $(1+|x|^2)^\frac{-(n-k)}{2(p-1)}$ is among the extremizers of the $L^p(\R^n)$-$L^q(\M)$ inequalities for the $k$-plane transform. They proved the $q=2$ case of the conjecture by relating the problem to the equivalent problem for the Hardy-Littlewood-Sobolev inequality.  As Lieb's work addressed both existence and uniqueness in this case, this proved that when $q=2$ all extremizers are of the form $c(\gamma+|x-a|^2)^{-(n+k)/2}$ for $c\in\C$, $\gamma>0$ and $a\in\R^n$. \\
\indent   For  $q=n+1$, the conjecture was proven for the Radon transform by Christ in \cite{C11} and for general $k$ by Drouot in \cite{D11}.  Christ also showed uniqueness: all extremizers of the endpoint inequality for the Radon transform are of the form $c(1+|\phi(x)|^2)^{-n/2}$ for $\phi$ an invertible affine endomorphism of $\R^n$, and all such functions are extremizers.  This paper extends the methods in \cite{C11} to the general $k$-plane transform.   Our main result is: 
	\begin{thm}\label{mainthm} 
		$f\in L^{(n+1)/(k+1)}(\R^n)$ is an extremizer of the inequality \eqref{main}  if and only if 
		\begin{equation*}
			f(x)=c(1+|\phi(x)|^2)^{-(k+1)/2}
		\end{equation*} 
		for some $c\in \C-\{0\}$ and some $\phi$ an invertible affine endomorphism of $\R^n$.
	\end{thm}
Uniqueness up to composition with affine maps is expected because of the symmetries of the problem. 
\begin{defn} Let $\varphi:\R^n\to\R^n$ be a function for which there exists a closed set $E\subset\R^n$ with $|E|=0$ such that $\varphi\in C^1(\R^n\setminus E)$ and $\varphi:\R^n\setminus E\to \R^n\setminus E$ is a bijection. Define $\mathcal{J}:L^p(\R^n)\to L^p(\R^n)$ by $\mathcal{J}f= |J_\varphi|^{1/p}(f\circ\varphi)$ where $|J_\varphi|$ is the Jacobian determinant of $\varphi$. Such a transformation is a symmetry of \eqref{main}  if  
$$ \|T_{k,n}\mathcal{J}(f)\|_{\lqm}=\|T_{k,n}f\|_{\lqm}.$$ 
\end{defn}
As  $\|f\|_{L^p(\R^n)}=\|\mathcal{J}(f)\|_{L^p(\R^n)}$, if $\mathcal{J}$ is a symmetry of \eqref{main} and $f$ is an extremizer of \eqref{main} then $\mathcal{J}(f)$ is also an extremizer of \eqref{main}. Composition with any invertible affine map is a symmetry of \eqref{main} (see \cite{D11}). That the set of  symmetries  of the endpoint inequality is in fact larger is crucial in the existence proof  in \cite{D11}, and is used to determine that $c(1+|x|^2)^{-(k+1)/2}$ is a radial extremizer. Sections 3 and 4 each give an interpretation of the additional symmetry.  \\
 \indent The proof of Theorem \ref{mainthm} has two main steps. The first, done by Drouot in \cite{D11}, is to show that extremizers exist and that $f=c(1+|x|^2)^{-(k+1)/2}$ is a radial nonincreasing extremizer. Drouot further proved the conditional result that if every extremizer of \eqref{main} has the form $f\circ\phi$ for $f$ a radial nonincreasing extremizer and $\phi$ an affine map, then all extremizers have the the form required in Theorem \ref{main}.  This paper concerns the second step, showing that the conditional step holds -- that any extremizer of \eqref{main} has the form $f\circ\phi$ for $f$ a radial nonincreasing extremizer and $\phi$ an affine map. \\
\indent Our analysis is modeled on that of Christ in \cite{C11}. The proof is similar to that for the Radon transform given in \cite{C11}, but the change in dimension presents two difficulties. The result in \cite{C11} relies on Burchard's theorem regarding cases of equality in the Riesz rearrangement inequality \cite{B96},\cite{Bthesis}. For the Radon transform the theorem proved in \cite{Bthesis} applies directly, but this result must be adapted before it applies for the $k$-plane transform case.  This is dealt with in \S 2.2. \\
\indent  Secondly, while in the case of the Radon transform it was known before \cite{C11} that extremizers of the endpoint inequality are smooth, in the general case they are not yet even known to be continuous. We modify the methods of \cite{C11} to apply to functions that are only assumed to be measurable. This takes the bulk of \S 2.3-2.4. Section 2.5 finishes the proof of Theorem \ref{mainthm}. \\
\indent Central to the analysis is a multilinear form (Drury's identity) that gives the $L^q$ norm of the $k$-plane transform.  A related multilinear form has been studied by Valdimarsson using similar methods in \cite{V12}.  As in Valdimarsson's case there is a certain amount of geometric invariance that allows us to immediately extend our result for the $k$-plane transform Euclidean space to the $k$-plane transform in elliptic space.  This transform was originally introduced by Funk \cite{F13}. See Helgason (for instance \cite{H99}) for the modern perspective. The question of  $L^p$-$L^q$ inequalities for the $k$-plane transform in elliptic space has been considered by Strichartz \cite{S81}, Christ \cite{C84}, and Drury \cite{D89}.  \\
\indent The $k$-plane transform in elliptic space is defined as follows.  Let $F$ be a function defined on $\mathcal{G}_{1,n}$, the set of lines through the origin in $\R^{n}$. Let $\pi\in\mathcal{G}_{k,n}$ be a $k$-plane passing through the origin  in $\R^{n}$. There is a unique probability Haar measure on the space of lines through the origin contained in $\pi$ analogous to that for $\mathcal{G}_{1,k}$. This measure will be denoted by $d\gamma_\pi$. The $k$-plane transform in elliptic space is given by
\begin{equation*}
T^E_{k,n}F(\pi)=\int_{\theta\subset\pi} F(\theta)\;d\gamma_{\pi}(\theta).
\end{equation*} 
Christ \cite{C84} proves that there exists a finite indeterminate constant $A_E$ such that for all $f\in L^p(\R^n)$,  
\begin{equation}\label{mainSph} 
		 \left( \int_{\mathcal{G}_{k,n}}\!\!\!\!\!\!|T^E_{k,n}F(\pi)|^{n}d\gamma(\pi) \right)^{1/n}\!\!\!\!\!\!\!\!
		 \leq A_E \left( \int_{\mathcal{G}_{1,n}}|F(\theta)|^{\frac{n}{k}}d\gamma(\theta) \right)^{\frac{k}{n}}.
	\end{equation} 
Assign coordinates on $\mathcal{G}_{1,n}$, losing a null set, by identifying each unit vector $\theta$ in the northern hemisphere with the line it spans. For a linear map $L$, $L(\theta)$ is the image of the unit vector $\theta$ under the map $L$. The main result of Section 3 is:
\begin{thm}   \label{SphThm} 
$F\in L^{\frac{n}{k}}(\mathcal{G}_{1,n})$ is an extremizer of the inequality \eqref{mainSph}  if and only if 
		\begin{equation*}
			F(\theta)=c\left|L(\theta)\right|^{-k} 
		\end{equation*} 
		 for some $c\in \C-\{0\}$ and some invertible linear endomorphism $L$ of $\R^{n}$. 
	\end{thm}

\indent  Section 4 concerns a third variant of the $k$-plane transform,  $T_{k,n}^\sharp$. Denote the space of $k\times (n-k)$ matrices by $Mat(k,n-k)$. Let $f:\R^n\to\C$, $A\in Mat(k,n-k)$ and $b\in\R^{(n-k)}$. Then $T_{k,n}^\sharp f$ is given by:
	\begin{equation*}
		T_{k,n}^\sharp f(A,b)=\int_{\R^k}f(x',A(x')+b)dx'. 
	\end{equation*}
We view $T_{k,n}^\sharp f(A,b)$ as a function on $\R^{(k+1)(n-k)}$ by identifying $Mat(k,n-k)\times\R^{(n-k)}$ with $\R^{(k+1)(n-k)}$ by first identifying $Mat(k,n-k)$ with $\R^{n-k}\times\ldots\times\R^{n-k}$.  As usual, equip $\R^{(k+1)(n-k)}$ with Lebesgue measure. The main result of section 4 is:
\begin{thm}\label{TTs} 
There exists a finite constant $A^\sharp\in \R_+$ such that for all $f\in L^p(\R^n)$ 
	\begin{equation}	\label{mains}
		 \left( \int_{\R^{k(n-k)}}\int_{\R^{n-k}}|T_{k,n}^\sharp f(A,b)|^qdAd b \right)^{1/q}\leq A^\sharp \|f\|_{\lp}.
	\end{equation}
Further, $f\in\lp$ is an extremizer of \eqref{mains} if and only if it is an extremizer of \eqref{main}. 
\end{thm}
Again, this is an extension of a result in \cite{C11} where Theorem \ref{TTs} is proved in the case that $k=n-1$.

\begin{nota} Where appropriate we identify functions $f\in L^p$ with the equivalence class of functions that are equal to $f$ almost everywhere.\\
\indent In Sections $2$ through $5$ the values of $p$ and $q$ will be fixed: $p=\frac{n+1}{k+1}$ and $q=n+1$. This convention is broken in Section 6, where more general $q$ are considered. \\
\indent We use $\R_+$ to denote the set of positive real numbers.  Let $E$ be a Lebesgue measurable set. $|E|$ will be the Lebesgue measure of $E$.  When $|E|>0$, $E^*$ will be the open ball centered at $0$ such that $|E|=|E^*|$. When $|E|=0$, $E^*$ will denote the empty set.  We use $\one_E$ to denote the indicator function of the set $E$. By the phrase ``$E=F$ up to a null set'' we mean that the symmetric difference of $E$ and $F$ has measure zero. The symmetric difference of two sets will be denoted by $\Delta$. Thus, $A\Delta B=(A\setminus B)\cup (B\setminus A)$.\\
\indent $\R^n=\R^{k}\times\R^{n-k}$ with coordinates $x'\in\R^{k}$ and $v\in\R^{n-k}$. We use $f_{x'}(v)=f(x',v)$ to discuss functions with the horizontal variable fixed. We also use $E(x',s)=\{v: f_{x'}(v)>s\}$ to denote the superlevel sets of these functions. Following the above notation, $E^*(x',s) $  is the open ball in $\R^{n-k}$ centered at 0 such that  $|E^*(x',s)|=|E(x',s)|$. \\  
\indent $\delta_{i,j}$ denotes the Kronecker delta. \\
\indent Lastly,  we have several notions of volume.  If $(x_0,\ldots,x_{k})$ is a generic point in $\R^{n(k+1)}$, $\pi(x_0,\ldots,x_k)$ will be the unique $k$-plane in $\R^n$ determined by $x_0,\ldots,x_k$ and $\voln(x_0,\ldots,x_k)$ will be the $k$-dimensional volume of the simplex determined by $x_0,\ldots,x_k$ in $\R^n$. We let $x'$ be the projection of $x\in\R^n$ onto $\R^k$ and  $\volk(x'_0,\ldots,x'_k)$ be the $k$-dimensional volume of the simplex formed by $x_0', \ldots, x_k'$ in  $\R^k$ . 
\end{nota} 
\end{section}  

\begin{section}{The $k$-plane transform in Euclidean space}

\indent Our analysis relies heavily on four results from the literature (which require three definitions to state). 
\begin{lem}[Drury's Identity, \cite{D83}] Let $f\in L^p(\R^n)$ be a nonnegative function. There exists $C\in\R_+$ depending only on $n$ and $k$ such that 
	\begin{multline*}
		 \left\|T_{k,n}f\right\|_{\lqm}^{q} \\ =C \int \prod_{i=0}^kf(x_i) \left( \prod_{i=k+1}^{n}\int_{\pi(x_0,\ldots,x_k)}\!\!\!\!\!\!\!\!\!\!\!\!\!\!\!\!\!\! f(x_i)\;d\sigma \right) 
		\voln^{(k-n)}(x_0,\ldots,x_k)\;dx_0\ldots dx_k
	\end{multline*}
where $\voln^{(k-n)}(x_0,\ldots,x_k)$ is the $k$-dimensional volume of the simplex determined by $x_0,\ldots,x_k$ in $\R^n$ raised to the power $(k-n)$ and $d\sigma$ is the surface measure on $\pi(x_0,\ldots,x_k)$.
\end{lem} 

\begin{defn}
	 Let $f$ be any measurable function on $\R^n$ such that all superlevel sets $\{x:|f(x)|>t\}$ for $t>0$ have finite measure.  Define $f^*$ the (symmetric nonincreasing) rearrangement of $f$ to be the function 
	\begin{equation*}
		f^*(x)=\int_0^\infty \one_{\{|f(x)|>t\}^*}(x)\;dt. 
	\end{equation*}
\end{defn}
$f_y^*(v)$ will denote the rearrangement of the function $f_y(v)=f(y,v)$ where $y\in\R^{n-k}$ is fixed.  It is a standard fact (see for instance \cite{LLtext}) that  $\|f\|_{L^p}=\|f^*\|_{L^p}$.

\begin{thm}[Brascamp, Lieb, and Luttinger's  rearrangement  inequality, \cite{BLL}] 
Let $f_i(x)$ for $1\leq i\leq m$  be nonnegative measurable functions on $\R^n$, and let $a_{i,j}$ for $1\leq i\leq m$ and $1\leq j\leq k$ be real numbers. Then 
\begin{equation*}
\int_{\R^{nk}}\prod_{i=1}^{m}f_i(\sum_{j=1}^{k}a_{i,j}x_j)\;dx_1\ldots dx_k\leq \int_{\R^{nk}}\prod_{i=1}^{m}f_i^*(\sum_{j=1}^{k}a_{i,j}x_j)\;dx_1\ldots dx_k.
\end{equation*} 
\end{thm} 

\begin{defn} Define 
	\begin{equation*}
		\mathcal{I}(E_0,\ldots,E_{m})= \int\left(\prod_{i=1}^{m}\one_{E_i}(x_i)\right) \one_{E_{0}}(x_1-\sum_{i=2}^{m}x_i)\;dx_1\ldots dx_{m}.
	\end{equation*}
\end{defn}

\begin{defn} A set of positive numbers $\{\rho_i\}_{i=0}^{m}$ is strictly admissible if they satisfy this 
	generalization of the triangle inequality:
	\begin{equation*}
	 \sum_{\stackrel{j=0}{j \neq i}}^{m}\rho_j>\rho_i\text{     for all } i\in[0,m]. \\
	\end{equation*}
\end{defn}

\begin{thm}[Burchard's theorem for indicator functions, \cite{B96}, \cite{Bthesis}]
	 Let $m\geq 2$. Let $ E_i$ for $i\in[0,m]$ be sets of finite positive measure in $\R^n$. Denote by  $\rho_i$ the radii of the $E_i^*$. If the family $\{\rho_i\}_{i=1}^m$ is strictly admissible and 
	\begin{equation*}
		\mathcal{I}(E_0,\ldots,E_{m})=\mathcal{I}(E^*_0,\ldots,E^*_{m})
	\end{equation*}
	then,  for each $i\in[0,m]$ there exist vectors $c_i\in\R^n$ and  numbers $\alpha_i\in\R_+$  such that $\displaystyle \sum_{i=1}^{m}c_i=c_{0}$, and there exists a fixed ellipsoid $\mathcal{E}\subset\R^{n}$ centered at the origin, such that  up to sets of measure zero 
	\begin{equation*}
		E_i=c_i+\alpha_i\mathcal{E}.
	\end{equation*}
\end{thm}

\begin{thm}[Drouot, \cite{D11}] Let $1\leq k\leq n-1$. Assume that any extremizer $f\in\lp$ for the $k$-plane transform inequality \eqref{main} can be written $f\circ \phi$ with $f$ a radial nonincreasing extremizer and $\phi$ an invertible affine map. Then any extremizer can be written 
	$$f=c(1+|\phi(x)|^2)^{-(k+1)/2}$$
with $c\in\C$ and $\phi$ an invertible affine map. 
\end{thm} 

The main idea of the proof of Theorem \ref{mainthm} is that if $f\in\lp$ is an extremizer of \eqref{main}, then $f$ produces equality in an inequality of the type addressed by Brascamp, Lieb, and Luttinger.   Although the cases of equality in general are not well understood, we are able to show that Burchard's work to applies to our case. This allows us to deduce that any extremizer is, up to composition with an affine map, a nonincreasing radial function. Our theorem then follows from Drouot's.  Our goal will be the following proposition:

\begin{prop}\label{radial}
	For any nonnegative extremizer $f\in L^p(\R^n)$ of \eqref{main}  there exists $\phi$ an invertible affine
	 transformation of $\R^n$, such that $f=F\circ\phi$ for $F$ some radial nonincreasing function $F:\R^n\to[0,\infty)$.
\end{prop}

\begin{proof}[Proof of Theorem 1 assuming Proposition 1] 
It is easy to see that if $f\in\lp$ is an extremizer of \eqref{main} then $f=c|f|$ for some $c\in\C-\{0\}$, thus it suffices to consider nonnegative extremizers. By Proposition 1, the conditions of Drouot's theorem are satisfied for all nonnegative functions, and  thus any extremizer can be written $f=c(1+|\phi(x)|^2)^{-(k+1)/2}$ for some $c\in\C$ and $\phi$ an invertible affine map. That any such function is an extremizer follows as $f=c(1+|x|^2)^{-(k+1)/2}$ is an extremizer, and invertible affine maps a symmetries of $\eqref{main}$ (\cite{D11}).  
\end{proof} 

\begin{subsection}{Direct Symmetrization } 
Following Christ's proof in \cite{C11}, we begin by reorganizing Drury's identity separating $\R^n$ into $\R^{k}\times\R^{n-k}$ with coordinates $x'\in\R^{k}$ and $v\in\R^{n-k}$.  After this change the inner integral will be of the form addressed by \cite{BLL} and, additionally, we may use the flexibility in varying the parameters in the outer integral to resolve some of the technical complications.
\begin{lem} \label{DL2}  Let $f\in L^p(\R^n)$ be a nonnegative function. There exists $C\in\R_+$ depending only on $n$ and $k$ such that 
 	\begin{multline*} 	
		 \|T_{k,n}f\|_{\lqm}^{q}= \\  
		C\int_{\R^{(n-k)(k+1)}}\!\!\!\!\!\!\!\!\!\!\!\!\!\!\!\!\!\!\!\! \!\!\!\! \volk^{(k-n)}(x'_0,\ldots, x'_k)\!\!\int_{\R^{k(n+1)}}\!\! \prod_{i=0}^kf(x'_i,v_i)\!\!\!\!  \prod_{i=k+1}^{n}\!\!\!\!
			f(x'_i,\sum_{j=0}^{k}\!b_{i,j}v_j)\;dv_0\ldots dv_k dx'_0\ldots dx'_{n}  
	\end{multline*} 
where $b_{i,j}$ are certain measurable real-valued functions of $x'_0,\ldots,x'_{k},x'_i$, $i$ and $j$. 
\end{lem}

\begin{proof}
	This is essentially a change of coordinates. Let $x_i=(x'_i,v_i)$ for $i\in[0,n]$. Take $x'_i$ to be an independent variable in $\R^k$ for each $i\in[0,n]$, and take  $v_i$ to be an independent variable in $\R^{(n-k)}$ for $i\in[0,k]$. Then for $i\in[k+1,n]$, $v_i$ will be determined by $x'_0, \ldots, x'_k$, $x'_i$, and $v_0,\ldots,v_k$ so that for $i\in [k+1, n]$,  each $(x_i',v_i)$ lies  in the $k$-plane spanned by $\{(x'_i,v_i)\}_{i=0}^k$. Specifically, let $A:\R^k\to\R^{n-k}$ be the unique affine map determined by $(k+1)$-tuple of equations $\{A(x_i')=v_i\}_{i=0}^k$. Then for $i\in[k+1,n]$, set $v_i=A(x_i')$. \\
\indent  Our goal is to express $d\sigma$ in terms of $dx_i'$ for $i\in[k+1,n]$. The parameterization above of $\pi(x_0,\ldots,x_k)$ takes the $k$-simplex in $\R^k$ spanned by $(x'_0,\ldots,x'_k)$ which has volume $\volk(x'_0,\ldots,x'_{k})$ to the $k$-simplex in $\R^n$ spanned by $(x_0,\ldots,x_{k})$ which has volume $\voln(x_0,\ldots,x_{k})$.  Therefore, for each $x_i$ with $i\in[k+1,n]$, $d\sigma(x_i)=\frac{\voln(x_0,\ldots,x_{k})}{\volk(x'_0,\ldots,x'_{k})}dx_i'$. As $n-k$ terms of this type appear in Drury's identity, the $\voln(x_0,\ldots,x_{k})$ terms cancel leaving
	\begin{multline*}
 \|T_{k,n}f\|_{\lqnm}^{n+1}= \\
		C\int\!\!\!\! \int\!\prod_{i=0}^k\!f(x'_i,v_i)\!\!\!\prod_{i=k+1}^{n}\!\!\!\!f(x'_i,A(x'_i))\volk^{(k-n)}(x'_0,\ldots, x'_k)\;
		dv_1\ldots dv_k dx'_0\ldots dx'_{n}.
	\end{multline*}
	Finally, a computation by Cramer's rule shows that for $i\in[k+1,n]$,  $A(x'_i)=\sum_{j=0}^{k}b_{i,j}v_j$ for coefficients $b_{i,j}$ given by 
	\begin{equation}\label{coef}
		b_{i,j}=\frac{\volk(x'_0,\ldots,x'_{j-1},x'_i,x'_{j+1},\ldots,x'_k)}{\volk(x'_0,\ldots,x'_k)}.
	\end{equation}
The formula \eqref{coef} gives $b_{i,j}=\delta_{i,j}$ if $0\leq i \leq k$. Define $b_{i,j}=\delta_{i,j}$ for all $0\leq i \leq k$. 
\end{proof}
 
The inner integral in Lemma \ref{DL2} becomes $$\int\prod_{i=0}^kf(x'_i,v_i)\prod_{i=k+1}^{n}f(x'_i,\sum_{j=0}^{k}b_{i,j}v_j)\;dv_0\ldots dv_k =\int\prod_{i=0}^{n}f_{x'_i}(\sum_{j=0}^{k}b_{i,j}v_j)dv_0\ldots dv_k .$$ 

\begin{defn} For $b_{i,j}$ with $i\in[0,n]$ and $j\in[0,k]$ depending on $(x_0',\ldots,x_n')$, given by \eqref{coef}, and $F_i:\R^{n-k}\to\R$ for all $i\in[0,n]$, let $\Tx$ denote the operator given by 
	\begin{equation*} 
		\Tx(F_0,\ldots,F_n)= \int\prod_{i=0}^{n}F_i(\sum_{j=0}^{k}b_{i,j}v_j)dv_0\ldots dv_k .
	\end{equation*}
\end{defn} 

\indent As the $b_{i,j}$ are real valued, by Brascamp, Lieb, and Luttinger's theorem
\begin{equation}\label{blleq}
	\Tx(F_0,\ldots,F_n) \leq\Tx(F_0^*,\ldots,F_n^*).
\end{equation}
 
\begin{lem}\label{Tx}
	For every nonnegative extremizer $f\in L^p(\R^n)$  of \eqref{main} and every symmetry $\mathcal{J}$ of \eqref{main}, for almost every $x'_0,\ldots,x'_n$ 
	\begin{equation*}
	\Tx(\mathcal{J}(f)_{x'_0},\ldots,\mathcal{J}(f)_{x'_n})=\Tx(\mathcal{J}(f)_{x'_0}^*,\ldots,\mathcal{J}(f)_{x'_n}^*). 
	\end{equation*}
\end{lem}

\begin{proof}
	As  $\mathcal{J}$ is a symmetry of \eqref{main}, $\mathcal{J}(f)$ is an extremizer of \eqref{main}, hence it suffices to consider $\mathcal{J}$ the identity transformation on $L^p(\R^n)$.  Multiplying both sides of \eqref{blleq} by $\volk(x'_0,\ldots,x'_k)^{(k-n)}$ gives 
	\begin{equation}\label{detT} 
	 \!\!\!\! \volk(x'_0,\!\smldots\!,x'_k)^{(k-n)}\Tx\!(f_{x'_0},\!\smldots\!,f_{x'_n})\!\leq\!\!\volk(x'_0,\!\smldots\!,x'_k)^{(k-n)}\Tx\!(f^*_{x'_0},\!\smldots\!,f^*_{x'_n})  .
	\end{equation}
\indent Let $f^\sharp(x,v)=f_x^*(v)$. Then integrating in each $x_i'$ shows 
	\begin{equation}\label{rrTkn}
		\|T_{k,n}f\|_{\lqm}^q\leq\|T_{k,n}f^\sharp\|_{\lqm}^q.
	\end{equation} 
Since $f$ is an extremizer, there is equality in \eqref{rrTkn}. Hence, there is equality in \eqref{detT} for almost every $x_0',\ldots,x_k'$. Multiplying by $\volk(x'_0,\ldots,x'_k)^{(n-k)}$, which is nonzero for almost every $x_0',\ldots,x_k'$,  proves the proposition. 
\end{proof}
Following Burchard \cite{B96} and Christ \cite{C11}, rather than work directly with $\Tx$, we further reduce to the case where $\Tx$ is applied to characteristic functions of superlevel sets of extremizers. This requires the layer cake decomposition of a function.
\begin{prop}[Layer cake decomposition (see for instance \cite{LLtext})] 
	If $f$ is a nonnegative measurable function, then 
	\begin{equation*} 
		f(x)=\int_0^\infty \one_{\{f(x)>t\}}(x)dt. 
	\end{equation*}
\end{prop}

To implement this reduction we will need a proposition parallel to Lemma \ref{Tx} for superlevel sets. 
\begin{prop}\label{rreqP} 
	 For every nonnegative extremizer $f$ of \eqref{main}, for almost every $x'_0,\ldots,x'_n$  and  almost every $s_0,\ldots, s_n$, 
\begin{equation} \label{rreqls}
\Tx(E(x'_0,s_0),\ldots,E(x'_n,s_n))= \Tx(E(x'_0,s_0)^*,\ldots,E(x'_n,s_n)^*)
\end{equation} 
where $E(x'_i,s_i)$ is shorthand for $\one_{E(x'_i,s_i)}$. 
\end{prop}

\begin{proof}
	Applying the layer cake decomposition to each $F_{x'_i}$,
	\begin{equation*}
		\Tx(F_{x'_0},\ldots,F_{x'_n})\!=\!\!\int_{(0,\infty)^{n\!+\!1}}\!\!\int_{(\R^{n\!-\!k})^{k\!+\!1}}\!\prod_{i=0}^n\one_{E(x'_i,s_i)}\!\!
		\left(\sum_{j=0}^{k}b_{i,j}v_j\right)\!\!\prod_{l=0}^k\!dv_l \!\!\prod_{m=0}^n\!\!ds_m
	\end{equation*}
	Similarly,
	\begin{equation*}
		\Tx(F^*_{x'_0},\ldots,F^*_{x'_n})\!=\!\!\int_{(0,\infty)^{n\!+\!1}}\!\!\int_{(\R^{n\!-\!k})^{k\!+\!1}}\!\prod_{i=0}^n
		\one_{E^*(x'_i,s_i)}\!\!
		\left(\sum_{j=0}^{k}b_{i,j}v_j\right)\!\!\prod_{l=0}^kdv_l\!\! \prod_{m=0}^n \!\!ds_m.
	\end{equation*}
	Again by the result of Brascamp, Lieb, and Luttinger in \cite{BLL}, 
	\begin{equation*}
		\!\!\! \int_{(\R^{n\!-\!k})^{k\!+\!1}}\!\prod_{i=0}^n\!\one_{E(\!x'_i,s_i\!)}\!\!
		\left(\!\sum_{j=0}^{k}b_{i,j}v_j\!\!\right)\!\!\prod_{l=0}^kdv_l\!\leq\!\!\!
		\int_{(\R^{n\!-\!k})^{k\!+\!1}}\!\prod_{i=0}^n\!\one_{E^*\!(\!x'_i,s_i\!)}\!\!
		\left(\!\sum_{j=0}^{k}b_{i,j}v_j\!\!\right)\!\!\prod_{l=0}^kdv_l\!\!\!.
	\end{equation*}
	Integrating in $s_i$ gives 
	\begin{equation*}
		\Tx(f_{x'_0},\ldots,f_{x'_n})\leq\Tx(f^*_{x'_0},\ldots,f^*_{x'_n}).
	\end{equation*}
	As equality holds here for almost every $x'_0,\ldots,x'_n$ and the product of characteristic functions is nonnegative, equality must hold in \eqref{rreqls} for almost every $x'_0,\ldots,x'_n$, for  almost every $s_0,\ldots, s_n$ . 
\end{proof}
\end{subsection}

\begin{subsection}{Inverse symmetrization for superlevel sets }

In \cite{C11},  Christ performs a change of variables and applies Burchard's Theorem (\cite{B96},\cite{Bthesis}) to conclude that the superlevel sets of the $f_{x_i}$ are intervals.  Here, because of the change in the relationship between the dimension and the number of functions, the result does not apply directly. Before applying Burchard's Theorem (\cite{B96},\cite{Bthesis}), we must first show that the extra $n-k$ functions are redundant given a modified admissibility condition and then apply a change of variables so that the functions, rather than the functional, depend on $b_{i,j}$. 
\begin{defn}
	A set of positive numbers $\{\rho_i\}_{i=0}^{n}$ is permissible with respect to $(x_0',\ldots,x'_n)$  if:
	\begin{eqnarray}
		\label{admA} \sum_{\stackrel{j=0}{j\neq i}}^{k+1}|b_{(k+1),j}|\rho_j>|b_{(k+1),i}|\rho_i& \text{  for all } i\in[0,k+1] \\
		\label{admB} \sum_{j=0}^{k}|b_{i,j}|\rho_j<\rho_i & \text{  for all } i\in[k+2,n] 
	\end{eqnarray}
	 where  the $b_{i,j}$ are determined for $i\in[0,n]$ and $j\in[0,k]$ by $x_0',\ldots,x'_n$ according to \eqref{coef} and $b_{(k+1),(k+1)}=1$.
\end{defn}

\begin{lem}\label{extraf}
	  For $i\in[0,n]$ let $E_i\subset \R^{n-k}$ be a set of finite positive measure. Let $\rho_i$ be the radius of $E_i^*$. If the set $\{\rho_i\}_{i=0}^n$ is permissible with respect to $(x_0',\ldots,x'_n)$   and $\Tx(E_0,\ldots,E_n)=\Tx(E^*_0,\ldots,E^*_n)$ then 
	\begin{equation*}
		\Tx(E_0,\ldots,E_{n})=\Tx(E_0,\ldots,E_{k+1},\R,\ldots,\R)
	\end{equation*}
	and 
	\begin{equation*}
		\Tx(E_0,\ldots,E_{k+1},\R,\ldots,\R)=\Tx(E^*_0,\ldots,E^*_{k+1},\R,\ldots,\R)
	\end{equation*}
\end{lem}

\begin{proof}
	By definition $\Tx(E^*_0,\ldots,E^*_n)=\int\prod_{i=0}^n\one_{E^*_i}(\sum_{j=0}^{k}b_{i,j}v_j)dv_0\ldots dv_k$.   Recall that $b_{i,j}=\delta_{i,j}$ if $i,j\in[0,k]$. Consider 
	\begin{equation*}
		\prod_{i=0}^k\one_{E^*_i}(v_i)\one_{E^*_l}(\sum_{j=0}^{k}b_{l,j}v_j). 
	\end{equation*}
	For $l\in[k+2,n]$, from the definition of permissibility \eqref{admB}, 
		$$ \rho_l>\sum_{j=0}^k|b_{l,j}|\rho_j.$$
	As $\rho_j$ is the radius of the open ball $E_j^*$ which is centered at the origin, it follows that for any choice of vectors $v_j\in E_j^*$, 
		$$\sum_{j=0}^k|b_{l,j}|v_j\in E_{l}^*.$$
	Therefore,
	\begin{equation*}
		\prod_{i=0}^k\one_{E^*_i}(v_i)\one_{E^*_l}(\sum_{j=0}^{k}b_{l,j}v_j)=\prod_{i=0}^k\one_{E^*_i}(v_i).
	\end{equation*}
	Because this holds for every $l\in[k+2,n]$, 
	\begin{equation*}
		\prod_{i=0}^k\one_{E^*_i}(v_i)\prod_{l=k+2}^{n}\one_{E^*_l}
			(\sum_{j=0}^{k}b_{l,j}v_j)=\prod_{i=0}^k\one_{E^*_i}(v_i).
	\end{equation*}
	Multiplying by $\one_{E^*_{k+1}}(\sum_{j=0}^{k}b_{l,j}v_j)$ yields, 
	\begin{equation*}
		\prod_{l=1}^{n}\one_{E^*_l}(\sum_{j=0}^{k}b_{l,j}v_j)=
		\prod_{i=0}^{k+1}\one_{E^*_i}(v_i).
	\end{equation*}
Multiply the right hand side by one in the form $\prod_{l=k+2}^{n}\one_{\R}(\sum_{j=0}^{k}b_{l,j}v_j)$ and integrate in $v_j$ for $j\in[0,k]$ to obtain
	\begin{equation} \label{TkTm}
		\Tx(E^*_0,\ldots,E^*_n)=\Tx(E^*_0,\ldots,E^*_{k+1},\R,\ldots,\R). 
	\end{equation}
	Now
		$$\prod_{i=0}^n\one_{E_i}(\sum_{j=0}^{k}b_{i,j}v_j)\leq\prod_{i=0}^{k+1}\one_{E_i}(\sum_{j=0}^{k}b_{i,j}v_j)$$
	because each term in the product is a characteristic function. Hence
		$$\Tx(E_0,\ldots,E_n)\leq\Tx(E_0,\ldots\!,E_{k\!+\!1},\R,\ldots,\R).$$
	Combining this with \eqref{TkTm} and the fact that $\Tx$ satisfies rearrangement inequalities yields
\begin{multline*} 
\Tx(E_0,\ldots,E_n)\leq\Tx(E_0,\ldots,E_{k\!+\!1},\R,\ldots,\R) \\ \leq
\Tx(E_0^*,\ldots,E_{k\!+\!1}^*,\R,\ldots,\R)=\Tx(E_0^*,\ldots,E_{n}^*)
\end{multline*} 
	Since by assumption $\Tx(E_0,\ldots,E_n)=\Tx(E_0^*,\ldots,E_{n}^*)$ equality must hold at every step.
\end{proof}

\begin{thm}[An adaption of Burchard's theorem for indicator functions] \label{ABT}
	 Let $ E_i$  be sets of finite positive measure in $\R^{n-k}$ for $i\in[0,n]$. Denote by $\rho_i$ the radius of $E_i^*$. If the family $\rho_i$ is permissible with respect to $(x'_0,\ldots,x'_n)$ and 
	\begin{equation*}
		 \Tx(E_0,\cdots,E_n)= \Tx(E^*_0,\cdots,E^*_n) 
	\end{equation*}
	then for each $i\in[0,k+1]$ there exist vectors $\beta_i\in\R^{n-k}$ and numbers $\alpha_i\in\R_+$ such that $\sum_{i=0}^{k}\beta_i=\beta_{k+1}$, and there exists an ellipsoid $\mathcal{E}$ which is centered at the origin and independent of $i$ such that,  up to null sets,
	\begin{equation*}
		b_{(k+1),i}E_i=\beta_i+\alpha_i\mathcal{E}
	\end{equation*}
	 where the $b_{i,j}$ are determined for $i\in[0,n]$ and $j\in[0,k]$ by $x_0',\ldots,x'_n$ according to \eqref{coef} and $b_{(k+1),(k+1)}=1$.
\end{thm}
\begin{proof} 
By Lemma \ref{extraf}, $ \Tx(E_0,\cdots,E_{k+1},\R,\ldots,\R)= \Tx(E^*_0,\cdots,E^*_{k+1},\R,\ldots,\R)$.   Set $y_0=b_{(k+1),0}v_0$ and $y_i= - b_{(k+1),i}v_i$ for $i\in[1,k]$.   Recall that $b_{i,j}=\delta_{i,j}$ if $i,j\in[0,k]$.
	\begin{align*}
		\Tx(E_0,\ldots,E_{k+1},\R,\ldots,\R)= \int\prod_{i=0}^{k}\one_{E_i}(v_i)\one_{E_{k+1}}(\sum_{j=0}^{k}b_{(k+1),j}v_j)dv_0\ldots dv_k\\
		= c \int\prod_{i=0}^{k}\one_{E_i}(b_{(k+1),i}^{-1}y_i)\one_{E_{k+1}}				(y_0-\sum_{j=1}^{k}y_j)dy_0\ldots dy_k. 
	\end{align*}
Therefore, 
	\begin{equation*}
		\Tx(E_0,\ldots,E_{k+1},\R,\ldots,\R)=c\mathcal{I}(E_{k+1},b_{(k+1),0}E_0,\ldots, b_{(k+1),k}E_k).
	\end{equation*}
The permissibility condition \eqref{admA} is precisely the requirement that the radii of $\{b_{(k+1),i}E_i^*\}_{i=0}^k\bigcup E^*_{k+1}$ are strictly admissible. Thus as the family $\rho_i$ is permissible with respect to $(x'_0,\ldots,x'_n)$,  Burchard's Theorem applied to $\{b_{(k+1),i}E_i^*\}_{i=0}^k\bigcup E^*_{k+1}$ gives the result.  
\end{proof} 

\end{subsection}

\begin{subsection}{Identifying $(n-k)$-cross sections of superlevel sets }

\begin{defn} To each nonnegative extremizer $f$ of \eqref{main}, associate a function $\rho(x',s)$ which is the radius of the ball $E^*(x',s)$. \end{defn} 

In this section we show that almost every $(n-k)$-cross section of almost every superlevel set is, up to a null set, an ellipsoid. The main step is to show that each such set of positive measure can be associated to an $(n+1)$-tuple of sets to which Burchard's theorem in the form of Theorem \ref{ABT} may be applied.    We construct such $(n+1)$-tuples predominantly following the proof of Lemma 5.4 in \cite{C11}. Our proof differs in that it is not yet known that extremizers are continuous, so we will rely on Lebesgue points of the function $\rho(x',s)$. The goal is:

\begin{prop} \label{lsE}  Let $f$ be any nonnegative extremizer of \eqref{main}. For almost every $x'\in\R^k$, for almost every $s\in\R_+$ the set $E(x',s)$ differs from an ellipsoid by a null set.
\end{prop}
\indent Before we wade into the proof we need  a few technical lemmas. 

\begin{lem}\label{Lpoints} For every nonnegative extremizer $f$ of \eqref{main} the associated function $\rho(x',s)$ is in $L^1_{loc}(\R^k\times\R_+) $.  In particular, almost every $(x',s)\in\R^k\times\R_+$ is a Lebesgue point of the function $(x',s)\to\rho(x',s)$. \end{lem} 
\begin{proof}  
This is a direct consequence of the definition, the observation that $f\in L^p(\R^n)$, and Fubini's theorem. 
\end{proof} 

\begin{lem}\label{nonneg}  Any nonnegative extremizer $f\in L^p(\R^n)$ of \eqref{main} satisfies $f(x)>0$ for almost every $x\in\R^n$. \end{lem}
The proof is deferred to the last section of the paper. 

\begin{lem}\label{us} Let $\{u_i\}_{i=0}^k$ be a set of pairwise-distinct unit vectors such that the volume of the simplex with vertices $0,u_1,\ldots,u_{j-1},u_{j+1},\ldots,u_k$ is independent of the choice of $j$.  Let $\tau>0$. If $x'_i=x'_{k+1}+\tau u_i$ for $i\in[0,k]$, then for all $j\in[0,k]$  $b_{(k+1),j}=\frac{1}{k+1}$. 
	\end{lem}
	\begin{proof}
Note that such $u_i$ exist in every dimension, take the vertices of a regular triangle, tetrahedron, etc. Take $u_i$ and $\tau$ as in the statement of the lemma.  Set $x'_i=x'_{k+1}+\tau u_i$ for $i\in[0,k]$. By choice of $u_i$   each of volumes ${\volk(x'_0,\ldots,x'_{j-1},x'_{k+1},x'_{j+1},\ldots,x'_k)}$ are equal.  Plugging this into \eqref{coef}, the definition of $b_{i,j}$, produces $b_{(k+1),j}=\frac{1}{k+1}$ for $j\in[0,k]$. 
  	\end{proof}

\begin{proof}[Proof of Proposition \ref{lsE} ] 
	Fix any $x'_{k+1}\in\R^k$ such that $f_{x'_{k+1}}$ is in $L^p(\R^{n-k})$,  $f_{x'_{k+1}}$ is positive almost everywhere and for almost every $s\in\R_+$, $(x'_{k+1},s)$ is a Lebesgue point of the function $(x',s)\to\rho(x',s)$. Almost every $x'\in\R^k$ satisfies all of these conditions: the first  because $f\in L^p(\R^n)$, the second by Lemma \ref{nonneg}, and the third by Lemma \ref{Lpoints}. Consider $s_{k+1}\in\R_+$. Either $\rho(x'_{k+1},s_{k+1})>0$ or $\rho(x'_{k+1},s_{k+1})=0$.  In the latter case, $|E(x_{k+1}',s_{k+1})|=0$ and the conclusion of Proposition \ref{lsE} is vacuously true. Hence it suffices to consider $s_{k+1}$ such that $\rho(x'_{k+1},s_{k+1})>0$. Fix some such $s_{k+1}\in\R_+$. \\
\indent Our goal is to construct a family of sets $\{S_i\subset(\R^k,\R_+): k+1\neq i\in[0,n]\}$ depending on $(x_{k+1}',s_{k+1})$ satisfying two conditions: first $|S_i|>0$  for $i\in[0,n]\setminus \{k+1\}$  and second, if for each $i\in[0,n]\setminus \{k+1\}$, $(x_i',s_i)\in S_i$ then $\{\rho(x_i',s_i)\}_{i=0}^n$ is permissible with respect to $(x_0',\ldots,x'_n)$.   Proposition \ref{rreqP}  guarantees that for almost every $(x_{k+1}, s_{k+1})\in\R^k\times\R_+$ for which such a family exists,  for almost every $(x_0',s_0),\ldots,(x'_k,s_k),(x'_{k+2},s_{k+2})\ldots,(x_n',s_n)$  equality in \eqref{rreqls}  holds in addition to permissibility. Applying Burchard's Theorem \ref{ABT} for superlevel to the sets  $E(x_0',s_0),\ldots,E(x_n',s_n)$ for which both equality and permissibility hold,  produces the desired conclusion. \\
\indent The first permissibility condition \eqref{admA} doesn't depend on $\rho_i$ for $i\in[k+2,n]$. Thus we begin by constructing  $\{S_i\}_{i=0}^k$ such that if $(x_i',s_i)\in S_i$ then $\{\rho(x_i',s_i)\}_{i=0}^{k+1}$ satisfies \eqref{admA}. \\
\indent Choose $r_{k+1}\in(0,s_{k+1})$ such that $(x'_{k+1},r_{k+1})$ is a Lebesgue point of the function $(x',s)\to\rho(x',s)$ and $\rho(x'_{k+1},r_{k+1})>\rho(x'_{k+1},s_{k+1})$. The first condition holds for almost every $r_{k+1}>0$ by choice of $x'_{k+1}$. The second must be satisfied by some $r_{k+1}$ as $f_{x_{k+1}'}(v) \in L^p(\R^{n-k})$ is almost everywhere positive, and thus larger superlevel sets always exist. \\
\indent Our strategy for constructing the family $\{S_i\}_{i=0}^k$ will be to find sets with positive measure of $(x',s)\in\R^k\times\R_+$ such that each $\rho(x',s)$ is approximately $\rho(x'_{k+1},r_{k+1})$ and the $b_{(k+1),j}(x_0',\ldots,x_{k+1}')$ for $j\in[0,k]$ are approximately equal to one another.\\
\indent Fix $\tolrho>0$ such that $4\tolrho<\min \left(\rho(x_{k+1}',r_{k+1})-\rho(x_{k+1}',s_{k+1}),\rho(x_{k+1}',s_{k+1})\right)$. This will be the tolerance in the size of superlevel  sets.\\
\indent Let $\mathcal{B}(\delrho) = B(x_{k+1}',\delrho)\times B(r_{k+1},\delrho)$.  Since $(x_{k+1}',r_{k+1})$ is a Lebesgue point of the function $(x',s)\to\rho(x',s)$, there is a $\delrho>0$ such that 
\begin{eqnarray*}
\frac{1}{|\mathcal{B}(\delrho)|}\mathlarger{\int}_{\mathcal{B}(\delrho) }|\rho(x',s)-\rho(x'_{k+1},r_{k+1})|dx'ds<\frac{\tolrho}{2(k+1)}.
\end{eqnarray*} 
Hence, 
\begin{equation*}
 \left|\left\{(x',s):|\rho(x',s)-\rho(x'_{k+1},r_{k+1})|>\frac{\tolrho}{2}\right\}\bigcap\mathcal{B}(\delrho) \right|<\frac{\left| \mathcal{B}(\delrho) \right|}{k+1}.
\end{equation*}
Therefore,  by the pigeonhole principle, it is possible to choose $\{u_i\}_{i=0}^k$ as in the statement of Lemma \ref{us}, $\tau\in(0,\delrho)$, and  $r_i\in B(r_{k+1},\delrho)$ for $i\in[0,k]$  such that $z'_i=x'_{k+1}+\tau u_i$ satisfy $|\rho(z'_i,s_i)-\rho(z'_{k+1},r_{k+1})|<\tolrho/2$  and each of the $(z'_i,r_i)$ are in turn Lebesgue points of $(x',s)\to\rho(x',s)$. Note that as computed in Lemma \ref{us} for $j\in[0,k]$, $b_{(k+1),j}(z_0',\ldots,z_k',x_{k+1}')=\frac{1}{k+1}$. Direct computation verifies that $\{\rho(z'_i,r_i)\}_{i=0}^{k}\cup\{\rho(x'_{k+1},s_{k+1})$ satisfy \eqref{admA}. \\
\indent Fix $\tolb<\tolrho \left((k+1)\rho(x_{k+1}',r_{k+1})\right)^{-1}.$ This will be the tolerance in the variation of the coefficients $b_{(k+1),j}$. For each $j\in[0,k]$ the function $b_{(k+1),j}:\R^{(n-k)(k+1)}\to\R$ is continuous, as it is a multilinear function of $\{x'_i\}_{i=0}^k$. Therefore, there exists $\delb>0$ such that if $x'_i\in B(z_i',\delb)$ for $i\in[0,k]$ and $y'_{k+1}\in B(x'_{k+1},\delb)$, then $|b_{(k+1),j}(z_0',\ldots,z_{k}',x_{k+1}')-b_{(k+1),j}(x_0',\ldots,x_{k}',y'_{k+1})|<\tolb $. \\
\indent Set $S_i=\{(x_i,s_i):|\rho(x'_i,s_i)-\rho(x'_{k+1},r_{k+1})|<\tolrho,x'_i\in B(z_i',\delb)\},$ for $i\in[0,k]$.  To see that for $i\in[0,k]$, $|S_i|>0$ recall that each $(z_i',r_i)$ is a Lebesgue point of the function $(x',s)\to\rho(x',s)$. Thus, there exists a small radius $\delta\in(0,\delb)$ such that for all $i\in[0,k]$  the condition $|\rho(z'_i,r_i)-\rho(x'_i,s_i)|<\tolrho/2$ is satisfied by at least half of the $(x'_i,s_i)$ such that $x'_i\in B(z'_i,\delta)$ and $|r_i-s_i|\leq\delta$. By the triangle inequality, such $(x'_i,s_i)$ also satisfy  $|\rho(x'_i,s_i)-\rho(x'_{k+1},r_{k+1})|<\tolrho$. \\
\indent  We now verify that any $(k+1)$-tuple $(x'_i,s_i)_{i=0}^k$ such $(x_i',s_i)\in S_i$ fulfills the permissibility condition.\\
\indent By Lemma \ref{us}, $b_{(k+1),j}(z_0',\ldots,z_{k}',x_{k+1}')=\frac{1}{k+1}$, therefore $b_{(k+1),j}=b_{(k+1),j}(x_0',\ldots,x_{k+1}')\in(\frac{1}{k+1}-\tolb,\frac{1}{k+1}+\tolb)  $. 
\begin{multline*} \sum_{j=0}^{k}|b_{(k+1),j}|\rho(x'_j,s_j) \geq (1-(k+1)\tolb)(\rho(x'_{k+1},r_{k+1})-\tolrho)   \\
\geq \rho(x'_{k+1},r_{k+1})-(k+1)\tolb\rho(x'_{k+1},r_{k+1})-\tolrho
\end{multline*}
As $\tolb< \frac{\tolrho}{(k+1)\rho(x'_{k+1},r_{k+1})}$ and $2\tolrho< \rho(x_{k+1}',r_{k+1})-\rho(x_{k+1}',s_{k+1})$ , 
$$ \sum_{j=0}^{k}|b_{(k+1),j}|\rho(x'_j,s_j) >\rho(x'_{k+1},s_{k+1}).$$
Fix any $i\in [0,k]$, then 
\begin{multline*}
 \sum_{\stackrel{j=0}{j \neq i}}^{k+1}|b_{(k+1),j}|\rho(x'_j,s_j) \geq  \frac{k-k\tolb}{k+1}\left(\rho(x'_{k+1},r_{k+1})-\tolrho\right)+\rho(x'_{k+1},s_{k+1}) \\
\geq \frac{k\rho(x'_{k+1},r_{k+1})}{k+1}+ \rho(x'_{k+1},s_{k+1}) -\frac{k}{(k+1)}(\tolrho+\tolb\rho(x'_{k+1},r_{k+1})).
\end{multline*} 
As $\tolb< \frac{\tolrho}{(k+1)\rho(x'_{k+1},r_{k+1})}$ and $\frac{k(k+2)}{(k+1)^2}<1$ , 
\begin{multline*}
 \sum_{\stackrel{j=0}{j\neq i}}^{k+1}|b_{(k+1),j}|\rho(x'_j,s_j) \geq   \frac{k\rho(x'_{k+1},r_{k+1})}{k+1}+ \rho(x'_{k+1},s_{k+1}) -\tolrho 
 \geq   \frac{k\rho(x'_{k+1},r_{k+1})}{k+1}+ 3\tolrho \\
\end{multline*} 
Additionally as $\tolb\leq\frac{\tolrho}{(k+1)\rho(x'_{k+1},r_{k+1})}<1$, and $k\geq 1$ , 
\begin{multline*}|b_{(k+1),j}|\rho(x'_j,s_j)<(\frac{1}{k+1}+\tolb)(\rho(x'_{k+1},r_{k+1})+\tolrho)\\ 
\leq \frac{\rho(x'_{k+1},r_{k+1})}{k+1}+\frac{3}{k+1}\tolrho \leq \frac{\rho(x'_{k+1},r_{k+1})}{k+1}+2\tolrho.\end{multline*} 

Therefore $$ \sum_{\stackrel{j=0}{j\neq i}}^{k+1}|b_{(k+1),j}|\rho(x'_j,s_j) >  |b_{(k+1),i}|\rho(x'_j,s_j).$$ 

s
\indent To prove the proposition, it remains to find a family $\{S_i\}_{i=k+2}^n$. Given the construction above, for $i\in[0,k]$ if $(x_i',s_i)\in S_i$, $\rho(x'_i,s_i)<\rho(x'_{k+1},r_{k+1})+\tolrho$. Moreover, if $i\geq k+1$ then $b_{i,j}(x'_0,\ldots,x'_n)=b_{(k+1),j}(x'_0,\ldots,x'_k,x'_i)$. Therefore $|b_{i,j}(x'_0,\ldots,x'_n)-b_{(k+1),j}(z'_0,\ldots,z'_k,x_{k+1})|<\frac{1}{k+1}+\tolb $. Hence, there exists $C'\in\R$ such that if $(x_j',s_j)\in S_j$ for $j\in[0,k]$ 
$$ \sum_{j=0}^k|b_{i,j}|\rho(x'_j,s_j)\leq C'.$$
 For each $i\in[k+2,n]$, set  $S_i=\{(x'_i,s_i): x'_i\in B(x'_{k+1},\delta)\text{ and } \rho(x'_i,s_i)>C'\}$.  $S_i$ for $i\in[k+2,n]$ has positive measure by positivity of the nonnegative extremizer $f$  (see Lemma \ref{nonneg}). Moreover, if $(x_i',s_i)\in S_i$ for $i\in[0,n]\setminus\{k+1\}$, then
 $$  \sum_{j=0}^k|b_{i,j}|\rho(x'_j,s_j)\leq C'<\rho(x'_i,s_i)$$ 
and hence \eqref{admB} is satisfied. 
\end{proof}

\end{subsection}

\begin{subsection} {Identifying $(n-k)$-cross sections part II: shared geometry} 

Thus far we have shown that almost all $(n-k)$-dimensional cross sections of the superlevel  sets of extremizers are ellipsoids up to null sets.  The next step is to show  that these elliptical cross sections almost always have the same geometry, i.e., they are translations and dilations of a single ellipsoid in $\R^{n-k}$. Further, we show that the translations are given by an affine function.\\
\indent We have not yet used the full strength of Burchard's theorem.  Applying Theorem \ref{ABT}:
\begin{lem}\label{sok}  For every nonnegative extremizer $f$ of \eqref{main}, for almost every $x'\in\R^k$, for almost every $s\in\R_+$, there exist an ellipsoid $\mathcal{E}(x')\subset\R^{n-k}$ centered at the origin, a vector $\gamma(x')\in\R^{n-k}$ and a number $\alpha(x',s)\in\R$ such that,  up to a null set, \begin{equation*} 
E(x',s)=\gamma(x')+\alpha(x',s)\mathcal{E}(x'). 
\end{equation*} 
\end{lem} 

\begin{proof} It is enough to prove the lemma for almost every $x'\in\R^{n-k}$, for almost any pair $s$ and $\tilde{s}$ such that both  $\rho(x_{k+1}',s)$ and $\rho(x_{k+1},\tilde{s})$ are nonzero.\\ 
\indent Take $x_{k+1}'\in\R^{n-k}$ satisfying the conditions of the construction of Proposition \ref{lsE}. Apply the construction, with $r_{k+1}$ chosen so that $\rho(x_{k+1}',r_{k+1})$ is greater than both $\rho(x_{k+1}',s)$ and $\rho(x_{k+1}',\tilde{s})$. This produces a family of measurable sets $\{S_i\subset\R^{n-k}\times\R_+:k+1\neq i\in[0,n]\}$, each with positive measure, such that if $(x'_i,s_i)\in S_i$ then $\{\rho(x_i',s_i)\}_{i=0}^n$ is permissible with respect to $(x'_0,\ldots, x'_n)$ both for $s_{k+1}=s$ and for $s_{k+1}=\tilde{s}$. By Proposition \ref{rreqP} , for almost every $x_{k+1}'\in\R^k$ and almost every pair $(s,\tilde{s})\in\R_+^2$, for almost every family $\{(x'_i,s_i): k+1 \neq i \in[0,n]\}$ with $(x'_i,s_i)\in S_i$, the $(n+1)$-tuple of sets $\{E(x'_i,s_i)\}_{i=0}^n$ produces equality in equation \eqref{rreqls}, both for $s_{k+1}=s$ and for $s_{k+1}=\tilde{s}$.\\ 
\indent  For any $(n+1)$-tuple of sets $\{E(x'_i,s_i)\}_{i=0}^n$ which produces equality in equation \eqref{rreqls} and is such that the set $\{\rho(x'_i,s_i)\}_{i=0}^n$ is permissible, Burchard's Theorem (Theorem \ref{ABT}) gives that for $i\in[0,k+1]$ there exist numbers $\alpha(x'_i,s_i)\in\R_+$, vectors $\beta(x'_i,s_i)\in\R^{n-k}$ satisfying  $\sum_{i=0}^{k}\beta(x'_i,s_i)=\beta(x'_{k+1},s_{k+1})$, and a fixed ellipsoid $\mathcal{E}(x'_i,s_i)$ which is centered at the origin and independent of $i$  such that, up to null sets,
\begin{equation*}
b_{(k+1),i}E(x'_i,s_i)=\beta(x'_i,s_i)+\alpha(x'_i,s_i)\mathcal{E}(x'_i,s_i).
\end{equation*}
Recall that $b_{(k+1),i}$ is given by \eqref{coef} for $i\in[0,k]$ and $b_{(k+1),(k+1)}=1$.  As $\mathcal{E}(x'_i,s_i)$ is determined by $\{(x'_i,s_i)\}_{i=0}^k$, $\mathcal{E}(x'_{k+1},s)=\mathcal{E}(x'_{k+1},\tilde{s})$. Set $\mathcal{E}(x'_{k+1})=\mathcal{E}(x'_{k+1},\tilde{s})$.  Similarly, $\beta(x'_{k+1},s_{k+1})=\sum_{i=0}^{k}\beta(x'_i,s_i)$, thus $\beta(x'_{k+1},s)=\beta(x'_{k+1},\tilde{s})$. Set  $\gamma(x'_{k+1})=\beta(x'_{k+1},s)$. \\
\indent With this terminology, for almost every $x'_{k+1}\in\R^{n-k}$, for almost every pair $s,\tilde{s}\in\R_+\times\R_+$,  both for $s_{k+1}=s$ and for $s_{k+1}=\tilde{s}$, up to a null set, 
\begin{equation*}
E(x'_{k+1},s_{k+1})=\gamma(x'_{k+1})+\alpha(x'_{k+1},s_{k+1})\mathcal{E}(x'_{k+1}).\qedhere
\end{equation*}
\end{proof}

Because superlevel sets are nested, this result extends to:  
\begin{prop} \label{alls}  For every nonnegative extremizer $f$  of \eqref{main}, for all $s\in\R_+$, for almost every $x'\in\R^k$, there exist an ellipsoid centered at the origin $\mathcal{E}(x')\subset\R^{n-k}$, a vector $\gamma(x')\in\R^{n-k}$, and  a number $\alpha(x',s)\in\R$ such that  $E(x',s)=\gamma(x')+\alpha(x',s)\mathcal{E}(x')$ up to a null set.
\end{prop}
\begin{proof}
Fix any $\tilde{s}\in\R_+$. Fix any $x'\in\R^k$ such that for almost every $s\in\R_+$, $E(x',s)=\gamma(x')+\alpha(x',s)\mathcal{E}(x')$ up to a null set. By Lemma \ref{sok}, this condition is satisfied by almost every $x'\in\R^k$.  Because superlevel sets are nested, for any sequence $s_n$ approaching $\tilde{s}$ from above, $E(x',\tilde{s})=\bigcup_{s_n}E(x',s_n)$. By our choice of $x'\in\R^k$, this sequence $s_n$ can be chosen such that for each $n\in\mathbb{N}$, $E(x',s_n)=\gamma(x')+\alpha(x',s_n)\mathcal{E}(x')$ up to a null set. As the union of a countable collection of null sets is a null set, 
$$E(x',\tilde{s})=\bigcup_{s_n}\gamma(x')+\alpha(x',s_n)\mathcal{E}(x')$$ 
up to a null set.  \\
\indent Set $\alpha(x',\tilde{s})=\lim_{n\to\infty}\alpha(x',s_n)$. This limit exists  because $\alpha(x',s_n)$ is nondecreasing and bounded as $n\to\infty$. The first condition holds because superlevel sets are nested. The second because $x'$ was chosen to satisfy the conditions of the construction in Proposition \ref{lsE} which require that $f_{x'_{k+1}}$ is in $L^p(\R^{n-k})$ and thus that each superlevel set of $f_{x'_{k+1}}(v)$ has finite measure. \\
\indent Therefore, up to a null set, 
\begin{equation*}E(x',\tilde{s}) = \gamma(x')+\alpha(x',\tilde{s})\mathcal{E}(x'). \qedhere
\end{equation*}
\end{proof} 

Our next goal is to show that there exists an ellipsoid centered at the origin $\mathcal{E}\subset\R^{n-k}$  such that for every $x'\in \R^k$, $\mathcal{E}(x')=\mathcal{E}$ and further that $\gamma(x')$ is an affine function.  A proof similar to that given for Lemma \ref{sok} holds if the extremizers are known to be continuous. However, for extremizers that are only known to be  measurable, there is an extra step.  We show that the results proved so far imply that any superlevel  set of an extremizer is convex up a null set and thus there exists a  representative of $f\in L^p(\R^n)$ whose  superlevel sets are convex.  This function will have the properties of continuous functions that are relevant to the proof. \\
\begin{defn} \label{alconvexdef}
A set $E$ is almost Lebesgue convex if for almost every pair $(x,y)\in E\times E$ the line segment $\overline{xy}\subset E$ up to a one-dimensional null set.  
\end{defn}
\indent  In Section \ref{sec:ALC} we prove Lemma \ref{alconvex}:  A set $E$ is almost Lebesgue convex if and only if there exists an open convex set $\Cv$ such that $|E\Delta \Cv|=0$ and in this case, $\Cv$ is the convex hull of the Lebesgue points of $E$.
\begin{prop}\label{convex}
For every nonnegative extremizer $f$ of \eqref{main}, for every $s\in\R_+$ the set $E_s=\{x\in\R^n:f(x)>s\}$  is an almost Lebesgue convex set. 
\end{prop} 
We will first show: 
\begin{lem}\label{alcSlices} 
For every nonnegative extremizer $f$ of \eqref{main}, for every $s\in\R_+$, for every $k$-plane $\theta\in\mathcal{M}_{k,n}$, for almost every $x'\in\theta$, and for almost every pair $(v_1,v_2)\in\theta^\perp\times\theta^\perp$ such that $x'+v_1\in E_s$ and $x'+v_2\in E_s$ , the line segment connecting $x'+v_1$ and $x'+v_2$ is contained in $E_s$ up to a one-dimensional null set. 
\end{lem} 

 Note that unlike most claims in this paper, which are of the almost everywhere variety, this result holds for every superlevel set and every  $k$-plane. 
\begin{proof} 
For any $k$-plane $\theta$ there is an affine map $A$ taking $\theta$ to $\R^k$. As $A$ is affine, the mapping $f\mapsto f\circ A$  is a symmetry of \eqref{main}.  Therefore $f\circ A$ is also an nonnegative extremizer of \eqref{main} and it suffices to consider the case where $\theta=\R^k\subset \R^n$. By Proposition \ref{alls}, for all $s\in\R_+$, for almost every $x'\in\R^k$, $E(x',s)$ is an ellipsoid, and hence convex, up to an $(n-k)$-dimensional null set, so the claim follows from the only if direction of Lemma \ref{alconvex}.  
\end{proof} 

\begin{proof}[Proof of Proposition \ref{convex}]

Factor $\R^n\times\R^n$ as the product $\mathcal{G}_{k,k+1}\times\R^k\times\R^{n-k}\times\R^{n-k}$, losing a null set, as follows.  For $x=(x_1,\cdots, x_n)$ write $x''=(x_1, \ldots, x_{k+1})$.  Almost every pair $(x'',y'')$ determines a line $\ell$ in $\R^{k+1}$.  There is a unique $k$-plane, $\theta$, in $\R^{k+1}$ that passes through the origin and is perpendicular to $\ell$ .  Let $x'\in\R^k$ denote the projection of $x$ onto $\theta$. As $\theta$ is perpendicular to $\ell$, the projection of $y$ onto $\theta$ is also $x'$. Let $v_x$ be the projection of $x$ onto $\theta^\perp$, the $(n-k)$-dimensional subspace perpendicular to $\theta$, and similarly for $v_y$. The 4-tuple $(\theta,x',v_x,v_y)$ completely specifies the pair $(x,y)$. \\
\indent By Lemma \ref{alcSlices} the set of  4-tuples $(\theta,x',v_x,v_y)$ such that the line segment connecting $x'+v_x$ and $x'+v_y$ is contained in $E_s$ up to a null set has full measure. Thus the set of $(x,y)$ such that $\overline{xy}\subset E_s$ up to a one-dimensional null set has full measure as well.
\end{proof} 

\begin{prop}\label{nicef}  For every nonnegative extremizer $f\in L^p(\R^n)$  of \eqref{main}, there exists $\tilde{f}\in L^p(\R^n)$ such that $\tilde{f}=f$ almost everywhere and every superlevel set of $\tilde{f}$ is open and convex.
\end{prop} 

\begin{proof} 
 Let $f$ be any nonnegative extremizer of \eqref{main}. Let $E_s=\{x:f(x)>s\}$. By Proposition \ref{convex} for every $s\in\R_+$,  the convex hull of the Lebesgue points of $E_s$, $\Cv_s$, is open and satisfies $|E_s\Delta \Cv_s|=0$. Define 
\begin{equation}\label{thegoodf} 
\tilde{f}(x)= \int_0^\infty \one_{\Cv_s}(x)ds.
\end{equation} 
Because $|E_s\Delta \Cv_s|=0$, $\tilde{f}(x)=f(x)$ almost everywhere . \\
\indent   Observe that the sets $\Cv_s$ are nested. Take  $r>t>0$.   $E_r\subset E_t$, thus the set of Lebesgue points of $E_r$ is contained in the set of Lebesgue points of $E_t$.  As $\Cv_r$ and $\Cv_t$ are the convex hulls of the Lebesgue points of $E_r$ and $E_t$ respectively, $\Cv_r\subset\Cv_t$.\\
\indent  For each $s\in\R_+$, define $\tilde{E}_s=\{x:\tilde{f}(x)>s\}$. Using \eqref{thegoodf} and that the sets $\Cv_s$ are nested, $\tilde{E}_s= \bigcup_{t>s} \Cv_t$. As the union of open sets is open, $\tilde{E}_s$ is open. Further, as the union of nested convex sets is convex, $\tilde{E}_s$ is also convex. 
\end{proof} 
\begin{cor} Any nonnegative extremizer $f$ of \eqref{main} agrees almost everywhere with a lower semi-continuous function. \end{cor} 

\begin{cor}\label{crho}  Let $f$ be a nonnegative extremizer of \eqref{main} whose superlevel sets are open and convex. For every $s\in\R_+$,  the function $x'\to\rho(x',s)$ is continuous on the interior of $\{x': \rho(x',s)\neq 0\}$. \end{cor} 

\begin{proof}
 Fix any $x'\in\R^k$ and $y'\in\R^k$ such that $|E(x',s)|\neq0$ and $|E(y',s)|\neq0$. By the Brunn-Minkowski inequality, 
\begin{equation*}
|tE(x',s) +(1-t)E(y',s) )|^{1/n}\geq t|E(x',s) |^{1/n}+(1-t)|E(y',s) |^{1/n}.
\end{equation*}
By convexity of the superlevel set $E_s$,  
\begin{equation*}
tE(x',s) +(1-t)E(y',s) ) \subset E (tx'+(1-t)y', s).
\end{equation*}
 Thus, 
\begin{equation*} 
\rho(tx'+(1-t)y',s)\geq  t\rho(x',s)+(1-t)\rho(y',s).
\end{equation*}
Hence $x'\to\rho(x',s)$ is concave on $\{x': \rho(x',s)\neq 0\}$. Using that a concave function on an open set is continuous, $x'\to\rho(x',s)$ is continuous on the interior of the set $\{x': \rho(x',s)\neq 0\}$.  
\end{proof}

\begin{prop} \label{samegeox} Let $f$ be a nonnegative extremizer of \eqref{main} whose superlevel sets are open and convex.   There exist an ellipsoid centered at the origin $\mathcal{E}\subset\R^{n-k}$, an affine function $\gamma(x')$, and numbers $\alpha(x',s)\in[0,\infty)$  such that for every $(x',s)\in \R^k\times\R_+$ satisfying $|E(x',s)|>0$
\begin{equation*}
	E(x',s)=\gamma(x')+\alpha(x',s)\mathcal{E}.
\end{equation*} 
\end{prop}

 \begin{proof}
By Proposition \ref{alls}, for all $s\in\R_+$, for almost every $x'\in\R^k$, there exist an ellipsoid centered at the origin $\mathcal{E}(x')\subset\R^{n-k}$, a vector $\gamma(x')\in\R^{n-k}$, and  a number $\alpha(x',s)\in\R$ such that  up to a null set, 
\begin{equation}\label{Estart} 
E(x',s)=\gamma(x')+\alpha(x',s)\mathcal{E}(x'). 
\end{equation}  
As $E(x',s)$ is open and convex, when $|E(x',s)|>0$ there is true equality in \eqref{Estart}, not just equality up to a null set. It remains to see that $\mathcal{E}(x')$ is independent of $x'$ and $\gamma(x')$ is an affine function.  \\
\indent  By the convexity established in Proposition \ref{nicef}, it suffices to show that for almost every $z'\in \R^k$ there exists some $\delta>0$ such that for almost every $x'\in B(z',\delta)$, $\mathcal{E}(z')=\mathcal{E}(x')$ and $\gamma(x')$ is almost everywhere equal to an affine function on $B(z',\delta)$.\\
\indent   Fix $z_{k+1}'\in\R^k$ satisfying the conditions of the construction in Proposition \ref{lsE} and take $s_{k+1}\in\R_+$ such that $z'_{k+1}$ is in the interior of $\{x': \rho(x',s_{k+1})\neq 0\}$. Such an $s_{k+1}$ always exists by positivity of nonnegative extremizers and convexity of each superlevel set. \\
\indent By essentially the same argument used for the construction in Proposition \ref{lsE}, there exist  $\delb>0$,  $\tolrho>0$, and  $\{S_i:k+1\neq i\in[0,n]\}$ such that if $|\rho(x'_{k+1},s_{k+1})-\rho(z'_{k+1},s_{k+1})|<\tolrho$ and $x'_{k+1}\in B(z'_{k+1},\delta_b)$ and $(x_i',s_i)\in S_i$ for $i\in[0,n]\setminus\{k+1\}$,  then $\{\rho(x_i',s_i)\}_{i=0}^n$ is permissible with respect to $(x'_0,\ldots,x'_{n+1})$. \\
\indent The only change required is in the definition of $\delb$. Whereas in the original proof $b_{(k+1),j}$ is viewed as a function of the $(k+1)$ variables $\{x'_i\}_{i=0}^k$ with $x'_{k+1}$ fixed, here $x'_{k+1}$ varies as well. Thus $b_{(k+1),j}$  is a function of the $(k+2)$ variables $\{x'_i\}_{i=0}^{k+1}$.  As this function is continuous, there exists $\delb>0$ such that if $x'_i\in B(z_i',\delb)$ for $i\in[0,k+1]$, then $|b_{(k+1),j}(x_0',\ldots,x_{k+1}')-b_{(k+1),j}(z_0',\ldots,z_{k}',z'_{k+1})|<\tolb $, where for $i\in[0,k]$, $z_i'$ is fixed as in Proposition \ref{lsE}. As there is an extra $\tolrho$ in the computation of permissibility  in Proposition \ref{lsE}, the same computation gives permissibility here. \\
 \indent By Corollary \ref{crho}, there exists $\delta_1>0$ such that for all $x'\in B(z_{k+1}',\delta_1)$, $|\rho(x',s_{k+1})-\rho(z_{k+1}',s_{k+1})|<\tolrho$. Set $\delta_{z_{k+1}'}=\min(\delta_1,\delb)$. Then, for every $x'_{k+1}\in B(z'_{k+1},\delta_{z_{k+1}'})$ if $(x_i',s_i)\in S_i$ for $i\in[0,n]\setminus\{k+1\}$,  $\{\rho(x_i',s_i)\}_{i=0}^n$ is permissible with respect to $(x'_0,\ldots,x'_{n+1})$. \\ 
\indent By Proposition \ref{rreqP}, for almost every $(z_{k+1}',s_{k+1})\in\R^k\times\R_+$ satisfying the conditions above, for almost every for almost every family $\{(x'_i,s_i): k+1 \neq i \in[0,n]\}$ with $(x'_i,s_i)\in S_i$,  for almost every  $x'_{k+1}\in B(z'_{k+1},\delta_{z_{k+1}'})$, the family $\{E(x_i',s_i)\}_{i=0}^n$ produces equality in \eqref{rreqls} and the family  $\{\rho(x_i',s_i)\}_{i=0}^n$ is permissible. Thus for almost every $z_{k+1}'\in\R^k$, there exist an $s_{k+1}\in\R^k$ and a family $\{(x'_i,s_i): k+1 \neq i \in[0,n]\}$ such that for almost every  $x'_{k+1}\in B(z'_{k+1},\delta_{z_{k+1}'})$, the $(n+1)$-tuple of sets $\{E(x_i',s_i)\}_{i=0}^n$ satisfies the conditions of Burchard's theorem.   \\
\indent Applying Burchard's theorem and Lemma \ref{sok} gives that there exist vectors $\beta(x'_i)\in\R^{n-k}$ satisfying $\sum_{i=0}^{k}\beta(x'_i)=\beta(x'_{k+1})$, numbers $\alpha(x'_i,s_i)\in\R_+$ and a fixed ellipsoid $\mathcal{E}(x'_i)$ which is centered at the origin and  independent of $i$, such that up to null sets, 
\begin{equation*}
b_{(k+1),i}(x'_0,\ldots,x'_{k+1})E(x'_i,s_i)=\beta(x'_i)+\alpha(x'_i,s_i)\mathcal{E}(x'_i). 
\end{equation*}

Therefore $\mathcal{E}(x'_{k+1})$ is determined by $\{x'_i\}_{i=0}^k$ and must be the same for almost every $x_{k+1}'\in B(z'_{k+1},\delta_{b,s_{k+1}})$.  \\
\indent Set\footnote{ Note that $\gamma(x'_{k+1})=\beta(x'_{k+1})$ so this definition agrees with the definition of $\gamma(x'_{k+1})$ given in Lemma \ref{sok}. }  $\gamma(x'_i)=\beta(x'_i)/b_{(k+1),i}(x'_0,\ldots,x'_{k+1})$.  
Therefore for almost every $x_{k+1}'\in B(z'_{k+1},\delta_{b,s_{k+1}})$, 
\begin{equation*}
\gamma(x'_{k+1})=\sum_{i=0}^{k}b_{(k+1),i}(x'_0,\ldots,x'_{k+1})\gamma(x'_i).
\end{equation*}
For $i\in[0,k]$, $b_{(k+1),i}(x'_0,\ldots,x'_{k+1})$ defined by \eqref{coef}) is an affine function of $x_{k+1}$, thus $\gamma(x'_{k+1})$ is as well. \\
\end{proof} 

\end{subsection}

\begin{subsection} {Proof of Proposition \ref{radial} } 
\begin{prop}\label{ksymradial} Let $f$ be a nonnegative extremizer of \eqref{main} whose superlevel sets are open and convex. Let $v\to f^\sharp(x',v)$ be the symmetric nonincreasing rearrangement of $v\to f(x',v)$ for each $x'\in\R^{(n-k)}$. Then there exist $ \gamma(x'):\R^k\to\R^{n-k}$ an affine function and $ L:\R^{n-k}\to\R^{n-k}$ an invertible linear map such that $f(x',L(v)+\gamma(x'))=f^\sharp(x',v)$.s \end{prop}
\begin{proof} Let $f\in L^p(\R^n)$ be any nonnegative extremizer of \eqref{main} whose superlevel sets are open and convex. By Proposition \ref{samegeox}, there exist an ellipsoid centered at the origin $\mathcal{E}\subset\R^{n-k}$, an affine function $\gamma(x')$, and numbers $\alpha(x',s)\in[0,\infty)$  such that for every $(x',s)\in \R^k\times\R_+$ satisfying $|E(x',s)|>0$
\begin{equation*}
E(x',s)=\gamma(x')+\alpha(x',s)\mathcal{E}.  
\end{equation*}
Let $ L:\R^{n-k}\to\R^{n-k}$ be the linear map taking  the unit ball to $\mathcal{E}$.  Thus for each $x'\in\R^k$, each superlevel set of the function $v\to f(x',L(v)+\gamma(x'))$ is a ball centered at the origin or the empty set. 
\end{proof} 

To prove Proposition \ref{radial}, we follow the proof in \cite{C11} for the Radon transform with modifications to allow for the change in dimension. 
This proof requires some notation from group theory. Let $\mathfrak{A}(n)$ denote the affine group and $O(n)$ denote the orthogonal group, each in $\R^n$. Similarly, let $O(n-k)$ denote the orthogonal group in $\R^{n-k}$. 

\begin{defn} Fix  $k\in[1,n-1]$. For  $\varphi\in O(n)$  define a scaled skew reflection associated to $\varphi$ to be any element of $\mathfrak{A}(n)$ with the form
\begin{equation*}
	\Phi_{\varphi}=\varphi^{-1}\psi^{-1}\mathcal{L}^{-1}R\mathcal{L}\psi\varphi
\end{equation*} 
 where $\psi(x',v)=(x', v+\gamma(x'))$ for $\gamma(x'):\R^k\to\R^{n-k}$ an affine mapping, $\mathcal{L}(x',v)=(x', L(v))$ for $L:\R^{n-k}\to\R^{n-k}$ an invertible linear map, and $R(x',v_1,\ldots,v_{n-k})= (x',v_1,\ldots,v_{n-k-1}, -v_{n-k}).$ \end{defn} 

\begin{lem}\label{ssr} For every nonnegative extremizer $f\in L^p(\R^n)$ of \eqref{main}, for each $\varphi\in O(n)$ there exists a scaled skew reflection associated to $\varphi$, $\Phi_{\varphi}$, such that$f\circ \Phi_{\varphi}= f$ almost everywhere. \end{lem} 

\begin{proof}   Given a nonnegative extremizer $f\in L^p(\R^n)$ of \eqref{main} and an orthogonal transformation $\varphi\in O(n)$, take $\gamma(x'):\R^k\to\R^{n-k}$ and $ L:\R^{n-k}\to\R^{n-k}$ to be the affine function and invertible linear map guaranteed by Proposition \ref{ksymradial} applied to the extremizer that agrees almost everywhere whose level sets are open and convex with $f\circ\varphi$.  Set $\mathcal{L}=(x',L^{-1}(v))$ and $\psi(x',v)=(x',v-\gamma(x'))$. Then by Proposition \ref{ksymradial},  $f\circ \Phi_{\varphi}= f$ almost everywhere.
\end{proof}

\begin{prop} \label{keystep}  Let $f:\R^n\to[0,\infty)$ be a  measurable function such that each superlevel  set  is convex and bounded.  Suppose $\{x:f(x)>0\}$ has positive Lebesgue measure and  for each $\varphi\in O(n)$ there exists a scaled skew reflection associated to $\varphi$, $\Phi_{\varphi}$, such that $f\circ \Phi_{\varphi}= f$ almost everywhere, then there exists $\phi\in\mathfrak{A}(n)$ such that $f\circ\phi=(f\circ\phi)^*$ almost everywhere. 
\end{prop}  
\begin{proof} 
For each $s\in\R_+$ set $E_s=\{x:f(x)>s\}$.  Let $G\subset\mathfrak{A}(n)$ be the subgroup of all $g\in\mathfrak{A}(n)$ such that $g(E_s)=E_s$ up to a null set for each $s\in\R_+$. As for some $s\in\R_+$ the set $E_s$ has positive measure and for each $s\in\R_+$ the set $E_s$ is bounded, $G$ is compact.  For each $\varphi\in O(n)$ there exists a scaled skew reflection associated to $\varphi$, $\Phi_{\varphi}$, such that$f\circ \Phi_{\varphi}= f$ and hence $\Phi_{\varphi}\in G$.  Any compact subgroup of $\mathfrak{A}(n)$ is conjugate by an element of $\mathfrak{A}(n)$ to a subgroup of $O(n)$ (see \cite{H78} pg 256). Thus, there exists $\phi\in\mathfrak{A}(n)$ such that for all $\varphi\in O(n)$, $\phi^{-1}\Phi_{\varphi}\phi\in O(n)$. Set $\tilde{\Phi}_{\varphi}=\phi^{-1}\Phi_{\varphi}\phi$.  \\
\indent Express $\R^n$ as $\R^{n-1}\times\R$. The transformation $\psi^{-1}\mathcal{L}^{-1}R\mathcal{L}\psi$ acts as the identity on $\R^{n-1}$, so $\tilde{\Phi}_{\varphi}$ acts as the identity on $\phi^{-1}\varphi^{-1}(R^{n-1})$. For a scaled skew reflection ${\Phi}_{\varphi}$, $\tilde{\Phi}_{\varphi}$ is an orthogonal reflection. Thus $\tilde{\Phi}_{\varphi}$ must be reflection about the hyperplane parallel to $\phi^{-1}\varphi^{-1}(R^{n-1})$ passing through origin.  As $\varphi$ ranges over $O(n)$,  the hyperplane parallel to $\phi^{-1}\varphi^{-1}(R^{n-1})$ passing through origin ranges over $\mathcal{G}_{(n-1),n}$. Thus the conjugated subgroup $\phi^{-1}G\phi$ contains a reflection about each $(n-1)$-dimensional subspace of $\R^n$.   These transformations generate the orthogonal group, so for each $s\in\R_+$, $\phi(E_s)$ is a convex set fixed under every orthogonal transformation. Therefore,  for each $s\in\R_+$, $\phi(E_s)$ must be a ball.\\
\end{proof} 

\begin{proof} [Proof of Proposition \ref{radial}]
For every nonnegative extremizer $f\in L^p(\R^n)$ of \eqref{main}, each superlevel set $E_s$ of $f$  is convex. As $f\in L^p(\R^n)$, each $E_s$ has finite measure. As a convex set with positive finite measure is bounded, for every $s\in\R_+$,  $E_s$ is bounded. Given this and Lemma \ref{ssr}, $\tilde{f}$ satisfies the conditions of Proposition \ref{keystep}. Hence $\tilde{f}=F\circ\phi$ for some radial function $F$ and $\phi$ some affine transformation of $\R^n$. As $\tilde{f}$ and $f$ are equal in $L^p$, this suffices. 
\end{proof}
\end{subsection}

\end{section} 

 \begin{section}{$k$-plane transform in elliptic space}
At the heart of this section is a correspondence between the $k$-plane transform in Euclidean space and the $(k+1)$-plane transform in elliptic space when $q=n+1$. This correspondence was originally observed by Drury  \cite{D89} for the $l$-to-$k$ plane transform and its elliptic analog.  Valdimarsson \cite{V12} uses a similar correspondence to extend his results on extremizers in $L^p(\R^n)$ for a multilinear form similar to the form which appears in Drury's identity to extremizers in $L^p(\Sp^{n+1}\cap\{ x_{n+1}>0\})$ for a corresponding version of the multilinear form.   \\ 
\indent Recall that the $(k+1)$-plane transform in elliptic space, defined in the introduction, is a bounded operator from $L^{\frac{n+1}{k+1}}(\mathcal{G}_{1,n+1})$ to $L^{n+1}(\mathcal{G}_{k+1,n+1})$.  Define a map from $\R^n$ to $\mathcal{G}_{1,n+1}$ by embedding $\R^n$ in $\R^{n+1}$ as $\{x_{n+1}=1\}$ and associating to each point $(x,1)$ the line it spans.  Parameterize $\mathcal{G}_{1,n+1}$ by $\theta\in\Sp^{n+1}\bigcap\{x_{n+1}>0\}$, losing a null set,  by associating unit vectors in the northern hemisphere with the lines they span. In these coordinates, the map described is a nonlinear projection onto the northern hemisphere:
$$\mathcal{S}(x)=\frac{1}{(1+|x|^2)^{1/2}}(x_1,\ldots,x_n,1).$$ 
 Let $d\sigma$ denote surface measure on the northern hemisphere and set $c_n=\int_{\Sp^{n+1}}\one_{\{\theta_{n+1}>0\}}(\theta)d\sigma.$ For this parametrization of $\mathcal{G}_{1,n+1}$ probability Haar measure is $c_n^{-1}\one_{\{\theta_{n+1}>0\}}(\theta)d\sigma$. \\
 \indent To a function $f\in L^\frac{n+1}{k+1}(\R^n)$, associate the function $F\in L^\frac{n+1}{k+1}(\mathcal{G}_{1,n+1})$ defined by

$$ F(\theta) =(\theta_{n+1})^{-(k+1)} f(\mathcal{S}^{-1}(\theta)).$$
 
Observe that $ c_n^{1/p}\|F\|_{L^p(\mathcal{G}_{1,n+1})} =
\|f\|_{L^p(\R^n)}$ when $p=\frac{n+1}{k+1}$.

\begin{lem} \label{STeq}There exists $C\in\R_+$ depending only on $n$ and $k$ 
 such that for every $f\in L^\frac{n+1}{k+1}(\R^n)$ and its associated function $F\in L^\frac{n+1}{k+1}(\mathcal{G}_{1,n+1})$ 
\begin{equation} 
|| T^E_{k+1,n+1}F(\theta) ||_{ L^{n+1}( \mathcal{G}_{k+1,n+1})}= C||T_{k,n}f||_{\lqnm}. 
\end{equation} 
\end{lem} 
\begin{proof} The nonlinear projection above also gives us a map from $\mathcal{G}_{k+1,n+1}$ to $\M$. For any $\pi\in\mathcal{G}_{k+1,n+1}$, let $\Pi\in\mathcal{M}_{k,n}$ be $\pi\cap\{x_{n+1}=1\}$ thought of as a $k$-plane in $\R^n$. Note that each line $\theta\in\pi$ corresponds to a point $\mathcal{S}^{-1}(\theta)\in\Pi$. Let $b(\Pi)$ denote the distance from $\Pi$ to the origin in $\R^{n+1}$. In \cite{D89}, Drury showed that there exists $c\in\R_+$ depending only on $k$ and $n$ such that Haar measure on $\mathcal{G}_{1,n+1}$, denoted $d\gamma$, is related to the natural product measure, denoted $d\mu$, on $\M$ by  
$$d\gamma(\pi)=c(b(\Pi))^{-(n+1)}d\mu(\Pi). $$  
The next step is to relate Haar measure on the set of linear subspaces contained in $\pi$, denoted $d\gamma_\pi$, to the natural product measure on the set of lines contained in $\Pi$, denoted $d\lambda_\Pi$. As each of the measures in question is invariant under rotation\footnote{To rotate the northern hemisphere, rotate the sphere and send any points of the northern hemisphere mapped into the southern hemisphere to their antipodal points.}, it is enough to consider  $\pi$ passing through the north pole of $\Sp^{n+1}$ and $\Pi$ passing through $(0,\ldots,0,b(\Pi))$.  In this case our map corresponds to division by $b(\Pi)$ followed by our original projection. Thus, 
$$\theta_{n+1}^{-(k+1)}d\gamma_{\pi}(\theta)=c_nb(\Pi)d\lambda_{\Pi}(x).$$
Therefore, 
\begin{align*}
T^E_{k+1,n+1}\left((\theta_{n+1})^{-(k+1)} f(\mathcal{S}^{-1}(\theta))\right)(\pi) &= \int_{\theta\subset\pi} (\theta_{n+1})^{-(k+1)} f(\mathcal{S}^{-1}(\theta)) d\gamma_\pi(\theta)\\
 &=c_n\int_{x\in\Pi} f(x) (b(\Pi))\;d\lambda_{\Pi}(x).
\end{align*}
Now, 
\begin{align*} 
|| T^E_{k+1,n+1}F ||_{ L^{n+1}( \mathcal{G}_{k+1,n+1})}^{n+1} &=\int_{\mathcal{G}_{1,n+1}}\left[T^E_{k,n}\left((\theta_{n+1})^{-(k+1)} f(\mathcal{S}^{-1}(\theta))\right)\right]^{n+1}d\gamma(\pi)\\
&=C\int_{\mathcal{M}_{k,n}}\left[\int_{x\in\Pi} f(x) (b(\Pi))\;d\lambda_{\Pi}(x)\right]^{n+1}(b(\Pi))^{-(n+1)}d\mu(\Pi)\\
 &= C|| T_{k,n}f ||_{ L^{n+1}( \mathcal{M}_{k,n})}^{n+1} \qedhere 
\end{align*}

\end{proof}

\begin{proof}[Proof of theorem \ref{SphThm}] 

By Lemma \ref{STeq}, there exists $C\in\R_+$ depending only on $n$ and $k$ such that for any $f\in\lp$ with $p=\frac{n+1}{k+1}$, 
$$ \frac{\|T_{k,n}f\|_{L^{n+1}(\M)}}{\|f\|_{\lp}}=C\frac{\|T^E_{k,n}F\|_{L^{n+1}(\mathcal{G}_{k+1,n+1})}}{\|F\|_{L^p(\mathcal{G}_{1,n+1})}}.$$
It follows immediately that $f\in L^p(\R^n)$ is an extremizer of \eqref{main} if and only if $F$ is an extremizer of \eqref{mainSph}.   \\
\indent By Theorem \ref{mainthm} any extremizer of \eqref{main} has the form $f(x)=c(1+|\phi(x)|^2)^{-(k+1)/2}$ where $\phi$ is an affine endomorphism of $\R^n$. It remains to compute the associated $F$.  Observe for any such $\phi$ there exists $L$, an invertible transformation of $\R^{n+1}$, such that $ (1+|\phi(x)|^2)= |L(x,1)|^2$.    Therefore, 
\begin{align*} F(\theta) &=(\theta_{n+1})^{-(k+1)} f(\mathcal{S}^{-1}(\theta))\\ &=c(\theta_{n+1})^{-(k+1)}(|L(\mathcal{S}^{-1}(\theta),1)|^2)^{-(k+1)/2} \\ & =c|L(\theta_1,\ldots,\theta_{n+1})|^{-(k+1)}. \qedhere
\end{align*} 
\end{proof} 
This perspective gives insight into the additional symmetry $J$ used in Christ \cite{C11} and Drouot's \cite{D11} work. Define $\mathcal{S}^*:L^p(\R^n)\to L^p(\mathcal{G}_{1,n+1})$ by $\mathcal{S}^*(f)=F$. Denote by $\sgn$  the standard sign function.  Set  $Jf(s,y)=|s|^{-k-1}f(s^{-1},s^{-1}y)$ and $RF(\theta)= F(\sgn(\theta_1)\theta_{n+1},\sgn(\theta_1)\theta_2,\ldots,\sgn(\theta_1) \theta_n,|\theta_1|)$. 
\begin{lem} For every $f\in L^\frac{n+1}{k+1}(\R^n)$, 
$$ \mathcal{S}^*Jf(\theta)= R\mathcal{S}^*f(\theta). $$
\end{lem} 
\begin{proof}
$$R\mathcal{S}^*f(\theta)=|\theta_1|^{-k-1}f\left(\frac{\sgn(\theta_1)\theta_{n+1}}{|\theta_{1}|},\frac{\sgn(\theta_1)\theta_2}{|\theta_{1}|},\ldots,\frac{\sgn(\theta_1)\theta_n}{|\theta_{1}|} \right).$$
Similarly, 
$$ \mathcal{S}^*Jf(\theta)= |\frac{\theta_1}{\theta_{n+1}}|^{-(k+1)}(\theta_{n+1})^{-(k+1)}f\left(\frac{\theta_{n+1}}{\theta_{1}},\frac{\theta_2}{\theta_{1}},\ldots,\frac{\theta_n}{\theta_{1}} \right).$$
As $\theta_{n+1}>0$, $ \mathcal{S}^*Jf(\theta)=R\mathcal{S}^*f(\theta)$ as claimed. 
\end{proof}

As the reflection $R$ is clearly a symmetry of \eqref{mainSph}, $J$ must be a symmetry of \eqref{main} by Lemma \ref{STeq}. 

\end{section}  

 \begin{section}{Another related family of operators} 
In this section we present yet another realization of the inequality \eqref{main}, this time for the operator $T_{k,n}^\sharp$ which was defined in the introduction. Recall that $T_{k,n}^\sharp$ takes functions on $\R^n$ to functions on $\R^{(k+1)(n-k)}$. 
\begin{lem} \label{TTsheq} Let $f\in L^p(\R^n)$ be a nonnegative continuous function. Then there exists $C\in\R_+$ depending only on $n$ and $k$ such that 
$$ \|T_{k,n}f\|_{\lqm}=C\|T_{k,n}^\sharp f\|_{L^{q}(\R^{(k+1)(n-k)})}.$$
\end{lem} 
The proof is a generalization of that used in \cite{C11} in the case $k=n-1$. 
\begin{proof} 
By Lemma \ref{DL2}, it suffices to show that for any nonnegative continuous function $f$, 
\begin{multline}
	\|T_{k,n}^\sharp f\|_{L^q(\R^{(k+1)(n-k)})}^q = \\  
\int\!\! \volk^{(k-n)}(x'_0,\ldots, x'_k)\!\!\int\!\! \prod_{i=0}^kf(x'_i,v_i)\!\!\!\!  \prod_{i=k+1}^{n}\!\!\!\!
			f(x'_i,\sum_{j=0}^{k}\!b_{i,j}v_j)\;dv_0\ldots dv_k dx'_0\ldots dx'_{n}.
\end{multline} 
\indent Let $c_{n-k}$ be the volume of the unit sphere in $(n-k)$ dimensions. Observe that 
\begin{multline*} 
T_{k,n}^\sharp f(A,b)=\int_{\R^k}f(x',A(x')+b)dx'\\
=\lim_{\epsilon\to0} \left(c_{n-k}\epsilon^{n-k}\right)^{-1}\!\!\!\! \int_{\R^k}\!\!\int_{\R^{n-k}}f(x',A(x')+b+t)\one_{|t|<\epsilon}\;dtdx'. 
\end{multline*} 
Taking $dA$ to be Lebesgue measure on the entries of $A$, and $db$ to be Lebesgue measure on $\R^{n-k}$, 
\begin{multline*} 
\!\!\int
\left(T_{k,n}^\sharp f(A,b)\right)^{n+1}dAdb  =\\
\int
\!\!\prod_{j=0}^k\!\left(\lim_{\epsilon\to0} \left(c_{n-k}\epsilon^{n-k}\right)^{-1}\!\!\!\!\int
\!\!\!f(x'_{\!j},A(x_{\!j}')\!+\!b\!+\!t_j)\one_{|t_j|<\epsilon}\,dt_jdx'_j\!\right)\!\!\!\!
\prod_{j=k+1}^n\!\!\!\! \left(\int_{\R^k}\!\!\!\!f(x_j',A(x_j')+b)dx_j'\!\!\right)
\!dAdb.
\end{multline*}
Apply the change of variables $s_j=Ax_j+b+t_j$ for $j\in[0,k]$ and Tonelli's theorem to obtain
\begin{multline}\label{exDs}
 \int \left(T_{k,n}^\sharp f(A,b)\right)^{n+1}dAdb  =
\int\prod_{j=0}^k f(x'_{\!j},s_j) 
\int \prod_{j=k+1}^n f(x_j',A(x_j')+b)\\
 \prod_{j=0}^k
\left(\lim_{\epsilon\to0} \left(c_{n-k}\epsilon^{n-k}\right)^{-1}\one_{|s_j-Ax_j+b|<\epsilon}\right)
dAdb
 \prod_{j=0}^kds_jdx_j'\prod_{j=k+1}^ndx_j'.
\end{multline}
Consider the inner integral, now viewing $\one_{|s_j-Ax_j+b|<\epsilon}$ as a cutoff function in $A$ and $b$.  Let $a_i$ be the $i$-th row of $A$ and $b_i$ be the $i$-th entry of $b$.  Let $L$ be the linear map such that $L(a_i,b_i)=(a_i\cdot x_j+b_i)_{j=0}^k$.  Then $L$ has a Jacobian $\mathcal{J}_L$ given by 
$$\mathcal{J}_L = \left(\begin{array}{cccc} x'_{0,1}&\cdots& x'_{0,k}&1\\ \vdots&&\vdots&\vdots\\ x'_{k,1}&\cdots& x'_{k,k}&1\\  \end{array}\right)=\volk(x'_0,\ldots,x'_k).$$
Let  $A_0$, $b_0$ such that $A_0x'_j+b_0=s_j$.  As $f$ is assumed to be continuous,
\begin{multline*} 
\lim_{\epsilon\to0} \left(c_{n-k}\epsilon^{n-k}\right)^{-1}\int \prod_{j=k+1}^n f(x_j',A(x_j')+b)
 \prod_{j=0}^k\left(\one_{|s_j-Ax_j+b|<\epsilon}\right)dAdb=\\
\volk(x'_0,\ldots,x'_k)^{k-n}\delta_{(A,b)(A_0,b_0)} \prod_{j=k+1}^n f(x_j',A(x_j')+b).
\end{multline*}
Substituting this into \eqref{exDs} gives the result.
\end{proof}

\begin{proof}[Proof of theorem \ref{TTs}] 
Using Lemma \ref{TTsheq} and standard approximation arguments, it follows that for any nonnegative function $f\in L^p(\R^n)$,  $ \|T_{k,n}f\|_{\lqm}=C\|T_{k,n}^\sharp f\|_{L^{q}(\R^{(k+1)(n-k)})}.$ As  $\| T_{k,n}^\sharp f\|_{L^q(\R^{(k+1)(n+1)})}\leq \| T_{k,n}^\sharp |f|\|_{L^q(\R^{(k+1)(n+1)})}$, it follows directly from Lemma \ref{TTsheq} and Theorem \ref{mainthm} that $T^\sharp_{k,n}$ is a bounded operator from $\lp$ to $L^q(\R^{(k+1)(n+1)})$.  Moreover, as $\| T_{k,n} f\|_{L^q(\M)}\leq \| T_{k,n} |f|\|_{L^q(\M)}$ as well,
$$\sup_{\{g:\|g\|_{L^p(\R^n)}\neq0\}} \frac{ \|T_{k,n}g\|_{\lqm}}{\|g\|_{L^p(\R^n)}}=\sup_{\{g:\|g\|_{L^p(\R^n)}\neq0,g>0\}} \frac{\|T_{k,n}g\|_{\lqm}}{\|g\|_{L^p(\R^n)}}.$$ 
 By Lemma \ref{TTsheq} there exists $C\in\R_+$ depending only on $n$ and $k$ such that
 $$\sup_{\{g:\|g\|_{L^p(\R^n)}\neq0,g>0\}} \frac{ \|T_{k,n}g\|_{\lqm}}{\|g\|_{L^p(\R^n)}}=\sup_{\{g:\|g\|_{L^p(\R^n)}\neq0,g>0\}} \frac{C\|T_{k,n}^\sharp g\|_{\lqm}}{\|g\|_{L^p(\R^n)}}.$$ 
Therefore, a nonnegative function $f\in\lp$ is an extremizer of \eqref{main} if and only if it is a nonnegative extremizer of \eqref{mains}. As any extremizer has the form $f=c|f|$ for some complex number $c$, this suffices. 
\end{proof} 

Again, the pseudo-conformal symmetry $J$ used to execute the method of competing symmetries in \cite{D11} is a natural symmetry of \eqref{mains}. Here, $J$ intertwines with changing the identification of $\R^{(k+1)(n-k)}=\R^{k(n-k)}\times\R^{(n-k)}\simeq Mat(k,n-k)\times\R^{(n-k)}=\{(A,b)\}$ by interchanging $b$ and the first row of $A$. Recall that $Jf=|s|^{-k-1}f(s^{-1},s^{-1}y)$. Let $A_b$ be the matrix $A$ with the first row replaced by $b$ and $a_1$ be the first row of $A$. Then define $R^\sharp F(A,b) =F(A_b,a_1)$. 
\begin{lem} For every $f\in L^p(\R^n)$, 
$$ T_{k,n}^\sharp Jf =R^\sharp T_{k,n}^\sharp f. $$ 
\end{lem}
\begin{proof}
\begin{align*} 
T_{k,n}^\sharp Jf(A,b) &=\int_{\R^k}(Jf(x',A(x')+b))dx' \\
&= \int_{\R^k}|s|^{-(k+1)}f(s^{-1},s^{-1}x',s^{-1}(A(s,x')+b)))dx' \\
&= \int_{\R^k}(|s|^{-(k+1)}f(s^{-1},s^{-1}x',(A(1,s^{-1}x')+s^{-1}b)))dx' \\.
\end{align*} 
Change  variables so that $t=s^{-1}$ and $w=s^{-1}x'$ to obtain
\begin{align*} 
T_{k,n}^\sharp Jf(A,b) & =  \int_{\R^k}(|f(t,w',(A(1,w')+tb)))dx' \\
& =  \int_{\R^k}(|f(t,w',(A_b(t,w')+a_1)))^{n+1}dx'\\
& =R^\sharp T_{k,n}^\sharp f.  \qedhere
\end{align*} 
\end{proof} 

\end{section} 

 \begin{section}{Almost Lebesgue convexity}\label{sec:ALC} 

Recall Definition \ref{alconvexdef}: A set $E$ is  almost Lebesgue convex if for almost every pair $(x,y)\in E\times E$ the line segment $\overline{xy}\subset E$ up to a one-dimensional null set.   Throughout this section $E_L$ will denote the set of Lebesgue points of a set $E$, and for any set $A$, $ch(A)$ will be the convex hull of $A$. 
\begin{lem}  \label{alconvex} A set $E$ is almost Lebesgue convex if and only if there exists an open convex set $\Cv$ such that $|E\Delta \Cv|=0$. In this case, $\Cv$ is the convex hull of the Lebesgue points of $E$. \end{lem}

We start with two lemmas that together prove the ``only if'' direction when $|E|>0$. 

\begin{lem}\label{ntuples} For any set $E$ with positive measure, if for almost every $(n+1)$-tuple $(x_1,\ldots,x_{n+1})\in E^{n+1}$, the convex hull $ch(x_1,\ldots,x_{n+1})\subset E$ up to an $n$-dimensional null set, then the convex hull of the Lebesgue points of $E$, $ch(E_L)$,  is an open convex set, and $|ch(E_L)\Delta E|=0$. 
\end{lem}  

\begin{proof} 
As $E_L\subset ch(E_L)$, $|E\setminus ch(E_L)|<|E\setminus E_L|=0$. Thus $|E\setminus ch(E_L)|=0$. \\
\indent It remains to show that $ ch(E_L)$ is open and $| ch(E_L)\setminus E|=0$. The main step is  to show that for each $(n+1)$-tuple of points $\{x_1,\ldots,x_{n+1}\}\in E_L^{n+1}$, there exists an open set $O_{\{x_1,\ldots,x_{n+1}\}}$, such that $ch(x_1,\ldots,x_{n+1})\subset O_{\{x_1,\ldots,x_{n+1}\}}\subset  ch(E_L)$ and $O_{\{x_1,\ldots,x_{n+1}\}}\setminus E$ is a null set.\\
\indent This claim implies the lemma as follows: By definition,
 $$\textstyle  ch(E_L)=\bigcup_{\{x_1,\ldots,x_{n+1}\} \in E_L^{n+1}}ch(x_1,\ldots, x_{n+1}).$$ 
As $ch(x_1,\ldots,x_{n+1})\subset O_{\{x_1,\ldots,x_{n+1}\}}$,
 $$\textstyle  ch(E_L)\subset \bigcup_{\{x_1,\ldots,x_{n+1}\}\in E_L^{n+1}}O_{\{x_1,\ldots, x_{n+1}\}}.$$
Similarly, because each $ O_{\{x_1,\ldots,x_{n+1}\}}\subset  ch(E_L)$,
 $$\textstyle  ch(E_L) \supset \bigcup_{\{x_1,\ldots,x_{n+1}\}\in E_L^{n+1}}O_{\{x_1,\ldots, x_{n+1}\}}.$$
Therefore, 
$$\textstyle  ch(E_L) = \bigcup_{\{x_1,\ldots,x_{n+1}\}\in E_L^{n+1}}O_{\{x_1,\ldots, x_{n+1}\}}.$$
As $ ch(E_L)$ is a union of open sets, it is open. Moreover, by the second countability of  $\R^n$, there exists $\{O_i\}$ a countable collection of $O_{\{x_1,\ldots,x_{n+1}\}}$,  such that $ ch(E_L)=\bigcup_{i=1}^\infty O_i$.  Thus $ ch(E_L)\setminus E\subset\bigcup_{i=1}^\infty (O_i \setminus E)$, which is a null set by countable additivity.  \\
\indent It remains to construct these $O_{\{x_1,\ldots,x_{n+1}\}}$. Begin by observing that given the conditions of the lemma, if $x\in E_L$ then there exists $\delta>0$ such that $B(x,\delta)\subset E$ up to an $n$-dimensional null set. Since $x\in E_L$, there exists $\delta'>0$ such that $|B(x,\delta')\cap E|\geq\frac{1}{n+1}|B(x,\delta')|$.  Applying the pigeonhole principle, there exists an $n$-tuple $\{x_i\}_{i=1}^{n+1}$  such that $x$ is in the interior of $ch(x_1,\ldots, x_{n+1})$ and $ch(x_1,\ldots, x_{n+1})\subset E$ up to an $n$-dimensional null set. Therefore, there exists $\delta>0$ such that $B(x,\delta)\subset ch(x_1,\ldots, x_{n+1})$, $B(x,\delta)\subset E$ up to an $n$-dimensional null set. \\
\indent  For any $(n+1)$-tuple of points $(x_1,\ldots, x_{n+1})\in E_L^{n+1}$, using the observation above, there exists a set of positive measure in $E_L^{n+1}$ of $y_1,\ldots, y_{n+1}$ such that $ch(x_1,\ldots, x_{n+1})\subset ch(y_1,\ldots, y_{n+1})$. By the hypothesis of the lemma,  for almost every such $(n+1)$-tuple, $ ch(y_1,\ldots, y_{n+1})\subset E$ up to a null set. Pick one of these $(n+1)$-tuples and take $O_{(x_1,\ldots, x_{n+1})}$ to be the interior of $ ch(y_1,\ldots, y_{n+1})$.  
\end{proof} 

\begin{lem}\label{ALChulls} If $E\subset\R^n$ is an almost Lebesgue convex set with positive measure and $m\in[2,n+1]$, then for almost every $m$-tuple $(x_1,\ldots, x_{m})\in E^{m}$, the convex hull $ch(x_1,\ldots, x_{m})\subset E$ up to an $(m-1)$-dimensional null set. 
\end{lem} 
\begin{proof} 
The proof proceeds by induction on $m$. If $E\subset\R^n$ is almost Lebesgue convex, then by definition the base case $m=2$  holds. Assume $m\in[2,n]$ and the statement is true for $m$. We seek to prove that for almost every $x_0$, for almost every $x_1,\ldots,x_m$, $ch(x_0,\ldots, x_{m})\subset E$ up to an $m$-dimensional null set.  \\
\indent Fix $x_0\in E$ such that for almost every $y$, $|\overline{x_0y}\setminus E|=0$. By almost Lebesgue convexity, it is enough to prove the statement for every such $x_0$.  Working in  polar coordinates centered at  $x$,   define $r_\theta = \sup\{r: |\overline{x_0(\theta,r)}\setminus E|=0\}$. Set  
$$S_{x_0} =\bigcup_{\theta\in\Sp^{n-1}}\overline{x_0(\theta,r_\theta)}.$$  
By the definition of $r_\theta$, $|S_{x_0} \setminus E|=0$. Moreover, because $|\overline{x_0y}\setminus E|=0$ for almost every $y$, $|E\setminus S_{x_0} |=0$. Therefore, $|S_{x_0} \Delta E|=0$. \\
\indent  Parameterize $m$-tuples in $\R^n$, losing a null set, by $(\pi,y, v_1, \ldots, v_m)$ where $\pi \in \mathcal{G}_{m-1,n}$, $y\in\pi^\perp$, $v_i\pi$ for $i\in[1,m]$.  Let $(\pi,y)$ denote the $(m-1)$-plane $\pi$ translated by $y$.  By the induction hypothesis, for almost every $\pi$, for almost every $y$, for  almost every $m$-tuple, $v_1,\ldots, v_m\in\pi^{m}$ such that $v_1 +y,\ldots,v_m+y \in E^{m}$,  $ch(v_1+y,\ldots,v_m+y)\subset E$  up to an $(m-1)$-dimensional null set.  \\
\indent Fix $\pi\in \mathcal{G}_{m-1,n}$ such that this condition holds. For almost every $y\in\pi^\perp$, $(\pi,y)\cap E$ satisfies  the conditions of Lemma \ref{ntuples}, hence there is a convex set $\Cv_{(\pi,y)}$ such that $|((\pi,y)\cap E)\Delta\Cv_{(\pi,y)}|=0$. For the null set of $y\in\pi^\perp$ for which such a set does not exist, let $\Cv_{(\pi,y)}$ be the empty set.  Set $$\Cv_\pi=\bigcup_{y\in\pi^\perp}\Cv_{(\pi,y)}.$$
Then $|\Cv_\pi\Delta E|=0$, and moreover, $|\Cv_\pi\Delta S_{x_0} |=0$. Thus for almost every $y\in\pi^\perp$, $|\Cv_{(\pi,y)}\Delta(S_{x_0} \cap(\pi,y))|=0$. Using that  $|\Cv_{(\pi,y)}\setminus S_{x_0} |=0$ and $S_{x_0} $ is star-shapped about $x_0$, for almost every $y\in\pi^\perp$, for  almost every $m$-tuple $v_1,\ldots, v_m\in\pi^{m}$ such that $v_1 +y,\ldots,v_m+y \in E^{m}$, $|ch(x_0,v_1+y,\ldots,v_m+y)\setminus S_{x_0} |=0$.  As $|S_{x_0} \Delta E|=0$, it follows that for  almost every $m$-tuple, $v_1,\ldots, v_m\in\pi^{m}$ such that $v_1 +y,\ldots,v_m+y \in E^{m}$, $ |ch(x_0,v_1+y,\ldots,v_m+y)\setminus E|=0$. 
\end{proof}

\begin{proof}[Proof of Lemma \ref{alconvex} ]
First consider the case that $|E|=0$.  Any null set is almost Lebesgue convex. The set of Lebesgue points for any null set is the empty set which is an open convex set equal to $E$ up to a null set. Hence the theorem holds when $|E|=0$.

The ``only if'' direction when $|E|>0$ is addressed by Lemmas \ref{ntuples} and \ref{ALChulls}.

To see the ``if'' direction, assume there exists an open convex set $\Cv$ such that $|E\Delta\Cv|=0$. As $|E|>0$, $| E\cap\Cv|>0$.  Fix any $x\in E\cap\Cv$. Take polar coordinates centered at $x$. For every $\theta\in\Sp^n$, define $r_\theta=\inf\{r:(\theta,r)\notin\Cv\}$. $r_\theta>0$ as $\Cv$ is open.  Further as $|\Cv\setminus E|=0$,  for almost every $\theta$, for every $0<r<r_\theta$ such that $(\theta,r)\in E$ the line segment in the direction $\theta$ up to distance $r$ is contained in $E$ up to a one-dimensional null set. As almost every point of $\Cv$ will be some $(\theta,r)$ such that this condition holds, it will hold for almost every point of $E$ as well. As almost every $x\in E $ is in $E\cap\Cv$, this suffices. 
\end{proof}

\end{section}

\begin{section}{Nonnegative extremizers are almost everywhere positive}

In order to show that nonnegative extremizers of \eqref{main} are positive almost everywhere, we instead prove a slightly more general statement. \\
\indent Note that all extremizers of \eqref{main} satisfy the Euler-Lagrange equation
\begin{equation}	\label{EL}
		f(x)=\lambda(T_{k,n}^*[(T_{k,n}f)^{q_0}])^{p_0}(x)
\end{equation}
where  $q_0=q-1$, $p_0=\frac{1}{p-1}$, $\lambda$ depends on $p,q,n,k$ and $f$,  and $T_{k,n}^*$ is the dual of the $k$-plane transform. \\

\begin{prop} If $f(x)\in L^p(\R^n)$ is a nonnegative solution of \eqref{EL} with $q\geq2 $, then either $f(x)>0$ for almost every $x\in\R^n$ or $f(x)=0$  for almost every $x\in\R^n$. \end{prop}

The proof relies on the following lemma. 

\begin{lem} For any nonnegative solution $f(x)\in L^p(\R^n)$ of \eqref{EL} with $q\geq2 $, \\
$f(x)\geq C(\lambda) (T_{k,n}^*T_{k,n}f(x))^{p_0q_0}$ almost everywhere. \end{lem}
\begin{proof}
	Let $d\theta$ be the unique Haar probability measure on $\mathcal{G}_{k,n}$ and $P(x,\theta^\perp)$ be the projection of $x$ onto $\theta^\perp$, the orthogonal complement  of $\theta$. Then, writing out $T^*_{k,n}$ explicitly, 
	$$f(x)=\lambda\left(\int_{\mathcal{G}_{k,n}}[T_{k,n}f(\theta,P(x,\theta^\perp))]^{q_0}d\theta\right)^{p_0}.$$
	 As $q_0=q-1=n\geq1$,  H\"older's inequality applies. 
	$$\int_{\mathcal{G}_{k,n}}g(\theta)d\theta\leq \left( \int_{\mathcal{G}_{k,n}}g(\theta)^{q_0}d\theta\right)^{1/q_0} 			\left(\int_{\mathcal{G}_{k,n}}1d\theta\right)^{1/q_0'}=C\left( 										\int_{\mathcal{G}_{k,n}}g(\theta)^{q_0}d\theta\right)^{1/q_0}.$$ 
Thus, $\int_{\mathcal{G}_{k,n}}[T_{k,n}f(\theta,P(x,\theta^\perp))]^{q_0}d\theta\geq\left(\int_{\mathcal{G}_{k,n}}(T_{k,n}f(\theta,P(x,\theta^\perp))d\theta\right)^{q_0}$. 
	 As $p=\frac{n+1}{k+1}\geq1$, $p_0=\frac{1}{p-1}>0$ and therefore
	$$f(x)\geq C\lambda\left(\int_{\mathcal{G}_{k,n}}[(T_{k,n}f(\theta,P(x,\theta^\perp))]\;d\theta\right)^{p_0q_0}.$$
	Again applying the definition of  $T^*_{k,n}$ proves the statement, with the qualification that as our function satisfies \eqref{EL} with equality in $L^p$, the statement holds only almost everywhere. 
	\end{proof}

\begin{proof}[Proof of Proposition] 
Writing out  $T_{k,n}^*T_{k,n}$ using Fuglede's formula \cite{F58}
$$f(x)\geq  C\lambda\left(\int f(y)|y-x|^{k-n}dx\right)^{p_0q_0}.$$
If there is a set of positive measure on which $f(x)=0$ then for some $x_0$ , 
$$ C\lambda(\int f(y)|y-x_0|^{k-n}dx)^{p_0q_0} =0.$$ 
$$ \int f(y)|y-x_0|^{k-n}dx =0.$$ 
As $|y-x_0|^{k-n}$ is positive except at $y=x_0$, $f(y)=0$ almost everywhere. 
\end{proof}

\end{section}

\begin{section} {Acknowledgments}
The author would like to thank Michael Christ for suggesting the problem and his guidance, and Alexis Drouot for insightful discussion. 
\end{section}

\bibliographystyle{amsplain}	
\bibliography{KPbib}

\end{document}